\documentclass[11pt,oneside,american,notitlepage]{article}
\usepackage[a4paper]{geometry}

\usepackage[utf8]{inputenc}

\usepackage{amssymb,amsmath,mathtools,xfrac,array}
\usepackage{dsfont}
\usepackage{amsthm}
\begingroup
    \makeatletter
    \@for\theoremstyle:=definition,remark,plain\do{%
        \expandafter\g@addto@macro\csname th@\theoremstyle\endcsname{%
            \addtolength\thm@preskip\parskip
            }%
        }
\endgroup

\usepackage[style=alphabetic,backend=biber,bibencoding=utf8,url=false,isbn=false,doi=false,maxbibnames=99]{biblatex}

\usepackage{authblk}

\usepackage{parskip}

\usepackage{subfiles}

\usepackage[svgnames]{xcolor}

\usepackage[hidelinks]{hyperref}

\numberwithin{equation}{section}
\newtheorem{theorem}{Theorem}[section]
\newtheorem{proposition}[theorem]{Proposition}
\newtheorem{lemma}[theorem]{Lemma}

\newtheorem{corollary}[theorem]{Corollary}
\theoremstyle{definition}

\theoremstyle{remark}
\newtheorem{remark}[theorem]{Remark}

\DeclareMathOperator{\supp}{supp}

\DeclareMathOperator{\I}{I}

\newcommand{\one}{\mathds{1}}
\newcommand{\E}{\mathbb E}
\newcommand{\F}{\mathcal F}
\newcommand{\B}{\mathcal B}

\newcommand{\PP}{\mathbb P}

\newcommand{\pQV}[1]{\langle#1\rangle}
\newcommand{\QV}[1]{\lbrack#1\rbrack}
\newcommand{\abs}[1]{\lvert#1\rvert}
\newcommand{\Abs}[1]{\left\lvert#1\right\rvert}
\newcommand{\norm}[1]{\lVert#1\rVert}

\newcommand{\llipnorm}[1]{\lVert#1\rVert_\mathrm{Lip}}

\newcommand{\R}{\mathbb{R}}
\newcommand{\Psim}{\stackrel{P}{\sim}}

\newcommand{\todo}[1]{}

\title{Concentration Inequalities for Additive Functionals: a
  Martingale Approach}
\author{Bob Pepin}

\usepackage[showzone=false,showseconds=false,useregional]{datetime2}
\DTMsettimestyle{iso}

\begin{document}
\maketitle
\begin{abstract}
  This work shows how exponential concentration inequalities for
  additive functionals of stochastic processes over a finite time
  interval can be derived from concentration inequalities for
  martingales. The approach is entirely probabilistic and naturally
  includes time-inhomogeneous and non-stationary processes as well as
  initial laws concentrated on a single point. The class of processes
  studied includes martingales, Markov processes and general square
  integrable càdlàg processes.  The general approach is complemented by
  a simple and direct method for martingales, diffusions and
  discrete-time Markov processes. The method is illustrated by deriving
  concentration inequalities for the Polyak-Ruppert algorithm, SDEs
  with time-dependent drift coefficients ``contractive at infinity''
  with both Lipschitz and squared Lipschitz observables, some
  classical martingales and non-elliptic SDEs. 
\end{abstract}

\section{Introduction}

In this work we consider concentration inequalities for additive
functionals of the form
\begin{equation*}
  \int_0^T X_t \, dt
\end{equation*}
where $X$ is a real-valued stochastic process. The methods we develop
apply to a broad class of processes, and we will give theorems and
examples that go beyond the classical setting of stationary Markov
processes. We will treat in depth the cases where $X$ is a martingale
or $X_t = f(t, Y_t)$ for a Markov process $Y$ and $f$ in an
appropriate class of functions. The concentration inequalities will be
derived for the additive functionals centered around their
expectation, which allows us to naturally treat non-stationary
processes such as time-inhomogeneous diffusions.

We proceed to give an overview of the main results. In
Section~\ref{sec:direct} we derive some representative results using
short, self-contained proofs based on direct calculations. First, we
show in Proposition~\ref{prop:direct:mart} that for any continuous
local martingale $X$ such that $X_0 = 0$ we have the following
concentration inequality for any $T, R > 0$:
\begin{equation*}
  \PP\left(\int_0^T X_u\,du \geq R\;; \int_0^T (T-u)^2\,d\QV{X}_u \leq
    \sigma^2\right) \leq \exp\left(-\frac{R^2}{2\sigma^2}\right).
\end{equation*}
To the author's knowledge, the systematic treatment of concentration
inequalities for additive functionals of martingales has not appeared
in the literature before.

We will then move on to solutions to SDEs of the form
\begin{equation*}
  dX_t = b(t, X_t) dt + \sigma dB_t,
\end{equation*}
with $X_0$ deterministic. In particular, denoting $P_{s,t}$ the
Markov transition operator associated to $X$, we will show in
Corollary~\ref{corr:sigmakappa} that
if there exist constants $c, \kappa > 0$ such that
\begin{equation*}
\abs{\sigma^\top \nabla P_{t,u} f(x)} \leq c e^{-\kappa (u-t)},
\quad x \in \R^n, 
0 \leq t \leq u
\end{equation*}
then we have the following Gaussian concentration inequality for all
$R, T > 0$:
\begin{equation*}
\PP \left(\frac1T \int_0^T f(u, X_u)\,du - \E\left[ \frac1T
\int_0^T f(u, X_u)\,du\right] \geq R\right) \leq
\exp\left(-\frac{\kappa^2 R^2 T}{2 c^2}\right).
\end{equation*}
The main novelty in the corollary is the treatment of time-inhomogeneous
SDEs and the method of proof. The Proposition from which the corollary
is derived also provides a novel, refined statement in terms of a bound of
$\abs{\sigma^\top \nabla P_{t,u} f}$ along trajectories of $X$.

Section~\ref{sec:direct} concludes with the case of discrete-time
processes in Proposition~\ref{prop:discrete}, again giving careful
consideration to the time-inhomogeneous case and controlling the relevant
quantities along trajectories of the process. Concretely, we show that
for a discrete-time stochastic process $X_t$ and function $f$ such
that
\begin{align*}
\abs{X_t - X_{t-1}} &\leq C_t, \quad t \geq 1, \\
\abs{P_{s,t} f(x) - P_{s,t} f(y)} &\leq \sigma_s\, (1-\kappa_s)^{t-s}\,
  \abs{x-y}, \quad x,y \in \R^n, 0 \leq s \leq t.
\end{align*}
we have
\begin{align*}
  \PP \left(\sum_{u=1}^{t} f(u, X_u) - \E\left[\sum_{u=1}^{t} f(u, X_u)\right] \geq R \;; \sum_{t=1}^T \frac{\sigma_t^2 \, C_t^2}{\kappa_t^2} \leq a^2\right)
  \leq \exp\left(-\frac{R^2}{8 a^2}\right).
\end{align*}
The careful treatment of the time-inhomogeneous case and control on
the level of trajectories enables in particular for the first time the
derivation of concentration inequalities for the Polyak-Ruppert
algorithm of the correct order in Section~\ref{sec:polyakruppert}
(concentration inequalities for the linear case were published
concurrently with this work in~\autocite{mou_linear_2020}).

Section~\ref{sec:main} is dedicated to a wide-ranging generalization of the
results from the previous section. In Subsection~\ref{sec:general} we
introduce a family of auxiliary martingales $Z^u_t = E^{\F_t} X_u$ and
show that for general square integrable processes, the concentration
properties of $\int_0^T X_u\, du$ are intimately linked to the
predictable quadratic covariation $\pQV{Z^u, Z^v}$ and jumps
$\Delta Z^u$ of the auxiliary martingales. In
Subsection~\ref{sec:general:martingales} we recover and extend the
martingale results from Section~\ref{sec:direct} to the discontinuous
setting using the general method. In
Subsection~\ref{sec:general:markov} we apply the general method to
general Markov processes and recover expressions for $\pQV{Z^u, Z^v}$
in terms of the squared field operator $\Gamma$, generalizing the
results on Markov processes from Section~\ref{sec:direct} to general
Markov processes on Polish spaces. All of the
results from the preceding subsections are novel in their generality. We
conclude Section~\ref{sec:main} with Subsection~\ref{subsec:martineq}
where we show how to incorporate arbitrary distributions of $X_0$ and
recall a number of martingale inequalities. The subsection also
includes in Corollary~\ref{corr:selfnormCD} a novel approach to obtain
Bernstein-type inequalities in some ``self-bounding'' cases.

The final Section~\ref{sec:examples} illustrates how to apply the
results from the preceding section on a number of concrete
cases. Subsection~\ref{sec:polyakruppert} contains the novel results
on Polyak-Ruppert mentioned above.

Subsection~\ref{example:sde} provides a concrete example of an SDE
case with explicit conditions on the drift and diffusion coefficients
as well as the observable function. Using known results on gradient
bounds for $P_{s,t}$ when the drift coefficient is ``contractive at
infinity''  characterized by constants $\rho, \kappa$, we show that for any initial
law $\nu$ satisfying a $T_1(C)$ transport inequality and Lipschitz
function $f$, we have
\begin{multline*}
  \PP\left(\frac{1}{T}\int_0^T f(X_t)\,dt - \frac1T \int_0^T \mu_t(f)
    \,dt \geq R +
       \rho \llipnorm{f} \frac{1-e^{-\kappa T}}{\kappa T} W_1(\mu_0, \nu)\right) \\ \leq
  \exp\left({-\frac{\kappa^2 R^2 T}{2 \rho^2 \llipnorm{f}^2 \left(1 +
      C\,\frac{1-e^{-\kappa T}}{T}\right)}}\right)
\end{multline*}
for a unique evolution system of measures $\mu_t$ (if the process has a stationary
measure $\mu$ then $\mu_t = \mu$ for all $t$).

Subsection~\ref{sec:examples:mart} provides some concrete examples using
classical martingales as integrands: Brownian motion, the compensated
Poisson process and compensated squared Brownian motion $B_t^2 - t$.

Subsections~\ref{sec:squaredOU} and~\ref{sec:squaredOU} treat the
cases where the integrand $X$ is either the squared Ornstein-Uhlenbeck
process or more generally the square of a Lipschitz function. These
cases go beyond the scope of most previously published approaches to
concentration inequalities for additive
functionals. The final Subsection~\ref{sec:examples:degenerate} presents a
simple case of a highly non-elliptic SDE, which yields easily to the
probabilistic methods presented here but is outside of the scope of
previous approaches based for example on Poisson equations.

\emph{About the literature.}  In the Markovian setting, our approach
is most closely related to the work of Joulin~\autocite{joulin_new_2009},
and we recover and extend the results from that work
(Propositions~\ref{prop:Markov} and~\ref{prop:sde:mu}). The cases of
martingales and general square integrable processes do not seem to
have been systematically studied in the literature.


Most previous results on concentration inequalities for functionals of
the form $S_t$ have been obtained for time-homogeneous Markov
processes using functional inequalities. The works
~\autocite{cattiaux_deviation_2008,guillin_transportation-information_2009-1,gao_bernstein-type_2014}
require the existence of a stationary measure and an initial
distribution that has an integrable density with respect to the
stationary measure. The same holds true for the combinatorial and
perturbation arguments in the classic
paper~\autocite{lezaud_chernoff_2001,}. In~\autocite{wu_transportation_2004,}
the authors establish concentration inequalities around the
expectation using stochastic calculus and Girsanov's theorem under
strong contractivity conditions. Some concentration inequalities for
inhomogeneous functionals have previously been established
in~\autocite{guillin_moderate_2001,}. A different approach using
renewal processes has been used in the
work~\autocite{locherbach_deviation_2012} to establish concentration
inequalities for functionals with bounded integrands.

For Markov processes, the mixing conditions in this work are most
naturally formulated in terms of bounds on either the Lipschitz
seminorm, gradient or squared field (carré du champs) operator of the
semigroup. Bounds on the Lipschitz seminorm are closely related to
contractivity in the $L^1$ transportation distance and can for
instance be found
in~\autocite{eberle_reflection_2015,assaraf_computation_2017} for
elliptic diffusions, in~\autocite{wu_gradient_2009,} for the
Riemannian case, in~\autocite{eberle_couplings_2019,} for Langevin
dynamics or in~\autocite{hairer_asymptotic_2011,} for stochastic delay
equations. See
also~\autocite{ollivier_ricci_2009,joulin_curvature_2010,} for the
discrete-time case and a large number of examples in both discrete and
continous time. Gradient estimates for semigroups can be obtained
using Bismut-type formulas, see for
example~\autocite{elworthy_formulae_1994,thalmaier_differentiation_1997,crisan_pointwise_2016},
the works \autocite{cheng_evolution_2018,cheng_exponential_2020} for
the non-autonomous case as well as the
textbook~\autocite{wang_analysis_2014,}. Finally, in terms of the
squared field operator, our mixing conditions are a relaxation of the
Bakry-Emery curvature-dimension
condition~\autocite{bakry_analysis_2014} since we allow for a
prefactor strictly greater than $1$.

\emph{Notation.} For a right-continuous process $(X_t)_{t\geq0}$ with left limits we write
$X_{t-} = \lim_{\varepsilon \to 0^+} X_{t-\varepsilon}$ and $\Delta X_t = X_t -
X_{t-}$. For a $\sigma$-field $\F$ and a random variable $X$, we denote $\E^{\F} X$
the conditional expectation of $X$ with respect to $\F$.

\section{Direct Approach}\label{sec:direct}
\subsection{Continuous Martingales}
In this and the following subsections we will establish concentration
inequalities by focusing on the cases of continuous local martingales,
continuous solutions to SDEs and discrete-time stochastic processes,
all with their initial law concentrated on a single point. In the
first two cases, we provide alternative and more direct proofs of
results that also follow from the general approach in Section~\ref
{sec:main}. Compared to the general approach, the direct proofs also
provide an explicit martingale representation for the additive
functionals under consideration. The discrete-time case introduces the
ideas of Section~\ref{sec:main} in a technically simpler setting. The
focus on a restricted class of processes keeps the proofs short and
self-contained and formal computations easily justifiable.  The
concentration inequalities presented here will be considerably
generalized to broader classes of processes, integrands and initial
laws in Section~\ref{sec:main}.

We begin with the case of a continuous local martingale. Even though
it is the simplest scenario in our framework, it has not been studied
systematically in the literature, which tends to concentrate on
concentration inequalities for additive functionals of stationary
Markov processes.

\begin{proposition}\label{prop:direct:mart}
  For a continuous local martingale $X$ such that $X_0 = 0$ we have
  the following concentration inequality for any $T, R > 0$:
\begin{equation*}
  \PP\left(\int_0^T X_u\,du \geq R\;; \int_0^T (T-u)^2\,d\QV{X}_u \leq
    \sigma^2\right) \leq \exp\left(-\frac{R^2}{2\sigma^2}\right).
\end{equation*}
Furthermore for $0 \leq t \leq T$
\begin{equation}\label{eq:martdecomp}
  \int_0^t X_u\,du = \int_0^t (T-u) \,dX_u - (T-t)X_t.
\end{equation}
\end{proposition}
\begin{proof}
Fix $T > 0$ and define a local martingale $M_t^T$ by
\begin{equation*}
  M^T_t = \int_0^t (T-u)\,dX_u.
\end{equation*}
Using integration by parts we obtain~\eqref{eq:martdecomp}
\begin{equation*}
  M^T_t = (T-t)X_t + \int_0^t X_u\,du
\end{equation*}
so that in particular $M^T$ coincides with $\int_0^\cdot X_u\,du$ at
time $T$:
\begin{equation}\label{eq:MTT}
  M^T_T = \int_0^T X_u\,du.
\end{equation}

It is well-known that for any continuous local martingale $M$ and $t\geq0$
\begin{equation}\label{eq:ddineq}
  \PP\left(M_t \geq R; \QV{M}_t \leq \sigma^2\right) \leq \exp\left(-\frac{R^2}{2\sigma^2}\right).
\end{equation}

Indeed, for $\lambda > 0$ let
$\mathcal{E}^\lambda(M)_t = \exp\left(\lambda M_t - \tfrac{\lambda^2}2 \QV{M}_t\right)$.
Since
$d\mathcal{E}^\lambda(M)_t = \mathcal{E}^\lambda(M)_t \,dM^T_t$
the process $\mathcal{E}^\lambda(M^T)_t$ is a positive local
martingale and therefore a supermartingale so that
$\E \, \mathcal{E}^\lambda(M^T)_T \leq \E \, \mathcal{E}^\lambda(M^T)_0 = 1$.
By Chebyshev's inequality, for any $R, \sigma^2 > 0$,
\begin{multline*}
  \PP\left(M_t \geq R; \QV{M}_t \leq \sigma^2\right)
  \leq \exp\left(-\lambda R + \tfrac{\lambda^2}2 \sigma^2\right) \; \E \left[
      \exp\left(\lambda M_t - \tfrac{\lambda^2}2 \sigma^2\right);
      \QV{M}_t \leq \sigma^2\right] \\ 
  \leq \exp\left(-\lambda R + \tfrac{\lambda^2}2 \sigma^2\right) \; \E
  \, \mathcal{E}^\lambda(M)_t \leq \exp\left(-\lambda R +
    \tfrac{\lambda^2}2 \sigma^2\right) \leq \exp\left(-\frac{R^2}{2\sigma^2}\right)
\end{multline*}
where $\lambda > 0$ was arbitrary and the last inequality follows by
optimizing over $\lambda$.

From the definition of $M^T$ and elementary properties of the
quadratic variation, we get
\begin{equation}\label{eq:QVcont}
  \QV{M^T}_t = \int_0^t (T-u)^2 d\QV{X}_u.
\end{equation}

By inserting~\eqref{eq:MTT} and~\eqref{eq:QVcont} into~\eqref{eq:ddineq}
we finally obtain
\begin{equation*}
  \PP\left(\int_0^T X_u\,du \geq R \;; \int_0^T (T-u)^2\,d\QV{X}_u \leq
    \sigma^2\right) \leq \exp\left(-\frac{R^2}{2\sigma^2}\right).
\end{equation*}
\end{proof}

\begin{corollary}
If $d\QV{X}_u \leq
\sigma^2 (T-u)^\alpha \,du$ for some $\sigma^2 > 0, \alpha \in (-3, +\infty)$ and all $0 \leq u
\leq T$ then for all $R > 0$
\begin{equation*}
  \PP\left(\frac{1}{T^{2+\alpha/2}}\int_0^T X_u\,du \geq R\right) \leq
  \exp\left(-\frac{(3+\alpha) R^2 T}{2 \sigma^2}\right).
\end{equation*}
\end{corollary}
\begin{proof}
We have
\begin{equation*}
  \int_0^T (T-u)^2\,d\QV{X}_u \leq \sigma^2 \int_0^T
  (T-u)^{2+\alpha}\,du = \frac{\sigma^2 T^{3+\alpha}}{3+\alpha}
\end{equation*}
  so that
  \begin{align*}
    \lefteqn{\PP\left(\frac{1}{T^{2+\alpha/2}}\int_0^T X_u \, du \geq R\right)}\quad \\
    &= \PP\left(\int_0^T X_u \, du \geq RT^{2+\alpha/2}\;; \int_0^T
      (T-u)^2\,d\QV{X}_u \leq \frac{\sigma^2 T^{3+\alpha}}{3+\alpha}\right)\\
    &\leq \exp\left(-\frac{(3+\alpha) R^2 T}{2 \sigma^2}\right).
  \end{align*}
\end{proof}


\begin{corollary}
  If $d\QV{X}_u \leq e^{-(T-u)}\,du$ for all $0 \leq u
  \leq T$ then for all $R > 0$
  \begin{equation*}
    \PP\left(\frac{1}{\sqrt{T}}\int_0^T X_u\,du \geq R\right) \leq
    \exp\left(-\frac{R^2 T}{4}\right).
  \end{equation*}
\end{corollary}
\begin{proof}
We have
\begin{equation*}
  \int_0^T (T-u)^2\,d\QV{X}_u \leq \int_0^T (T-u)^2 e^{-(T-u)}\,du = 2
  - (T^2 + 2T + 2)e^{-T} \leq 2
\end{equation*}
and the result follows.
\end{proof}

\begin{remark}
The continuity assumption primarily serves to keep the
exposition self-contained by providing a concise proof of the concentration
inequality~\eqref{eq:ddineq} for local martingales. The method extends
to discontinuous local martingales by making use of the more advanced martingale
inequalities described further below in Section~\ref{subsec:martineq}.
\end{remark}

\begin{remark}
  Compared to Proposition~\ref{prop:Martingale}, the result presented
  here is a priori more general because it applies to local
  martingales, whereas Proposition~\ref{prop:Martingale} is stated for
  martingales. On the other hand, the process $M^T$ constructed in the
  proof of this section is also only a local martingale, whereas the
  general method yields a process $M^T$ which is a martingale by
  construction. Under the hypotheses of
  Proposition~\ref{prop:Martingale} that $X_t$ is a true martingale
  with $\E X_t^2 < \infty$ for all $t \geq 0$, the present result also
  yields a true martingale, since then
  $\E\QV{M^T}_t \leq T^2\, \E\QV{X}_t < \infty$ for all $t \geq 0$.
\end{remark}

\subsection{Stochastic Differential Equations}
In this section, we consider the case of a solution $X$ to the
following elliptic SDE on $\R^n$:
\begin{equation}\label{eq:SDEX}
  dX_t = b(t, X_t) dt + \sigma dB_t, \quad X_0 = x_0
\end{equation}
for $x_0 \in \R^n$ fixed, $b \colon [0, \infty) \times \R^n \to \R^n$
locally bounded, once differentiable in its first argument and twice
differentiable in the second with bounded first derivative, $\sigma$ a
real $n\times n$ matrix such that $\sigma \sigma^\top$ is
positive definite and $B$ a standard $n$-dimensional Brownian motion.

For a bounded function $f$ on $[0, \infty) \times \R^n$, twice continuously
differentiable, and $0 \leq t \leq u$ we define the two-parameter
semigroup associated to $X$ and $f$,
\begin{equation*}
  P_{t,u} f(x) = \E[f(u, X_u) | X_t = x].
\end{equation*}

If the coefficients of~\eqref{eq:SDEX} are time-homogeneous, we have
$P_{t,u} f = P_{0,u-t}f = P_{u-t}f$, where the latter is just the usual
(time-homogeneous) Markov semigroup.

The result in Corollary~\ref{corr:sigmakappa} below is known in the
time-homogeneous case, and can for example be deduced using the
arguments in~\autocite{joulin_new_2009,}. The refined formulation in
Proposition~\ref{prop:direct:Markov} and the inclusion of the
time-inhomogeneous setting is new.

\begin{proposition}\label{prop:direct:Markov}
For all $T, R, c > 0$ and $f \colon [0, \infty) \times
\R^n \to \R$, bounded and twice continuously
differentiable, we have
\begin{align*}
  & \PP \left(\int_0^T f(u, X_u)\,du - \E \int_0^T f(u, X_u)\,du \geq
    R\;; \int_0^T \abs{\sigma^\top \nabla R_t^T(X_t)}^2\,dt \leq c^2
  \right) \\
  & \quad \leq \exp\left(-\frac{R^2}{2 c^2}\right)
\end{align*}
with
\begin{equation*}
  R_t^T(x) = \int_t^T P_{t,u}f(x)\,du, \quad x \in \R^n, 0 \leq t \leq
  T.
\end{equation*}
Furthermore we have the decomposition
\begin{equation*}
  \int_0^t f(u, X_u)\,du = \int_0^t \nabla R_s^T(X_s) \cdot
  \sigma\,dBs - R_t^T(X_t), \quad 0 \leq t \leq T.
\end{equation*}
\end{proposition}
\begin{proof}
For $T > 0$ fixed define a martingale $M^T$ by
\begin{equation*}
  M^T_t = \E^{\F_t} \int_0^T f(u, X_u)\,du, \quad 0 \leq t \leq T
\end{equation*}
where $\F_t = \sigma(\{X_s\}_{s\leq t})$ is the natural filtration of
$X$.  Using Fubini's theorem with the bounded integrand $f$, the
fact that $f(u, X_u)$ is $\F_t$-measurable for all $u \leq t$ and the Markov property
we get
\begin{align}\label{eq:sde:1}
  M^T_t &= \int_0^T \E^{\F_t} f(u, X_u)\,du = \int_0^t f(u, X_u)\,du +
  \int_t^T \E^{\F_t} f(u, X_u)\,du \notag \\ &= \int_0^t f(u, X_u)\,du + R_t^T (X_t)
\end{align}
with
\begin{equation*}
  R_t^T(x) = \int_t^T P_{t,u} f(x) du.
\end{equation*}
By our regularity assumptions on $f$ and the coefficients
of~\eqref{eq:SDEX}, we have the Kolmogorov backward equation
$(\partial_t + L) P_{t,u} f = 0$ where $L$ denotes the
infinitesimal generator of $X$.
It follows from standard properties of the integral that
\begin{equation}\label{eq:sde:2}
  (\partial_t + L) R_t^T(x) = -P_{t,t} f(x) + \int_t^T (\partial_t + L) P_{t,u} f(x)
  du = -f(t, x), \quad t > 0, x \in \R^n.
\end{equation}

Using the Itô formula, we get
\begin{align*}
  dR^T_t(X_t)
  &= (\partial_t + L)R_t^T(X_t) \, dt + \nabla R_t^T(X_t) \cdot
                \sigma \, dB_t \\
  &= -f(t, X_t) \,dt + \nabla R_t^T(X_t) \cdot \sigma \, dB_t.
\end{align*}
By~\eqref{eq:sde:1} we also have
\begin{equation}
  dR^T_t(X_t) = -f(t, X_t)\,dt + dM^T_t
\end{equation}
so that we can identify
\begin{equation}
  dM^T_t =\nabla R_t^T(X_t) \cdot \sigma \, dB_t. \label{eq:sde:3}
\end{equation}
From elementary properties of the stochastic integral and Brownian motion
\begin{equation*}
  d\QV{M^T}_t = \abs{\sigma^\top \nabla R_t^T(X_t)}^2 \, dt.
\end{equation*}
In the proof of Proposition~\ref{prop:direct:mart} we saw that for any
continuous local martingale $M$ with $M_0 = 0$ we have
\begin{equation*}
  \PP(M_t \geq x, \QV{M}_t \leq y) \leq \exp\left(-\frac{x^2}{2 y}\right).
\end{equation*}

Now the result follows by applying this inequality to
$(M^T_T - M^T_0)$ and noting that
$M^T_T = \int_0^T f(u, X_u)\,du$ and
$M^T_0 = \E^{\F_0} \int_0^T f(u, X_u)\,du = \E \int_0^T f(u,
X_u)\,du$, by our assumption that $X_0$ is
concentrated on a single point $x_0 \in \R^n$.
\end{proof}

\begin{corollary}\label{corr:sigmakappa}
If there exist constants $c, \kappa > 0$ such that
\begin{equation*}
\abs{\sigma^\top \nabla P_{t,u} f(x)} \leq c e^{-\kappa (u-t)}
\quad x \in \R^n, 
0 \leq t \leq u
\end{equation*}
then we have the following Gaussian concentration inequality for all
$R > 0, T > 0$:
\begin{equation*}
\PP\left(\frac1T \int_0^T f(u, X_u)\,du - \E\left[ \frac1T
\int_0^T f(u, X_u)\,du\right] \geq R\right) \leq
\exp\left(-\frac{\kappa^2 R^2 T}{2 c^2}\right).
\end{equation*}

\end{corollary}
\begin{proof}
  
From our assumption we can estimate for all $x \in \R^n, 0 \leq t \leq
T$
\begin{align*}
\abs{\sigma^\top \nabla R_t^T (x)} & \leq \int_t^T\abs{\sigma^\top
                                      \nabla P_{t,u} f(x)} \, du \leq c
  \int_t^T e^{-\kappa (t-u)} du \leq \frac{c}{\kappa}
\end{align*}
where the regularity assumptions on the coefficients of~\eqref{eq:SDEX}
ensure well-posedness of the Kolmogorov backward equation and by
extension sufficient smoothness to differentiate under the integral
for the first inequality.

Now the result follows directly from
Proposition~\ref{prop:direct:Markov} since
\begin{multline*}
  \PP \left(\frac1T \int_0^T f(u, X_u)\,du - \E\left[ \frac1T
      \int_0^T f(u, X_u)\,du\right] \geq R\right)
  \\ = \PP \left(\int_0^T f(u, X_u)\,du - \E\left[
\int_0^T f(u, X_u)\,du\right] \geq RT\;; \int_0^T \abs{\sigma^\top \nabla
R_t^T(X_t)}^2\,dt \leq \frac{c^2}{\kappa^2} T
  \right).
\end{multline*}
\end{proof}

\begin{remark}
  The approach presented here, based on stochastic calculus, requires
  the existence of an Itô-type formula and the well-posedness of the
  Kolmogorov backward equation. The approach essentially operates on
  the level of the individual trajectories of $X$, and in return
  yields an explicit martingale representation for $\int_0^T f(u, X_u)\,du$. In contrast,
  the results on Markov processes from Proposition~\ref{prop:Markov}
  rely on knowing the quadratic variation of a certain auxiliary
  martingale, for which it is sufficient to have a characterization of
  $X$ in terms of a martingale problem.
\end{remark}

\begin{remark}
  Regarding the law of $X_0$, in contrast to the analytic approaches
  present in the literature, the case where $X_0$ is concentrated on a
  single point is the most natural for our probabilistic approach. In
  particular, this case is outside of the scope of many existing
  results in the literature such as
  ~\autocite{guillin_transportation-information_2009,gao_bernstein-type_2014},
  which require the law of $X_0$ to be
  absolutely continuous with respect to a stationary measure of the
  process. In Section~\ref{subsec:martineq} we will see how to deal
  with more general initial laws in the context of the approach
  presented here.
\end{remark}

\begin{remark}
  We avoid the question of stationarity altogether by deriving
  concentration inequalities of $\tfrac1T \int_0^T f(u, X_u)\,du$
  centered around its expectation, as opposed to the usual formulation
  which shows concentration around $\int f d\mu$ for a stationary
  measure $\mu$. In Section~\ref{example:sde} we will see an example
  of how to derive concentration inequalities around $\int f d\mu$.
\end{remark}

\begin{remark}
  The essential ingredient in the proof is equation~\eqref{eq:sde:2},
  which states that $(t, x) \mapsto \int_t^T P_{t,u}f(x)\,du$ is a
  solution to the following PDE in $g$ on $[0, T] \times \R^n$:
  \begin{align*}
    L g + \partial_t g &= -f.
  \end{align*}
  This observation is not new and the solution $g$ notably features
  prominently in the classic book on martingale problems by Stroock
  and Varadhan
  \autocite{stroock_multidimensional_2006,}. However, the
  application to additive functionals seems to be new, if not entirely
  surprising. Indeed, a popular approach to additive functionals of Markov
  processes
  \autocite{cattiaux_central_2012,guillin_transportation-information_2009-1,}
  involves the solution $g$ to the Poisson equation on $\R^n$,
  \begin{align*}
    Lg &= -f.
  \end{align*}
  Under certain strong ergodicity conditions, requiring in particular
  the stationarity of $X$, a solution to the Poisson equation can be
  shown to exist and is then given by the well-known resolvent formula
  as $x \mapsto \int_0^\infty P_{0,u} f(x)\,du$. In contrast, the
  definition of $R_t^T(x) = \int_t^T P_{t,u}f(x)\,du$ makes sense in a
  very general setting.

\end{remark}

\subsection{Discrete Time Markov Process}\label{sec:discrete}
In this section, we consider the case of a discrete-time Markov chain
in order to build some probabilistic intuition for our assumptions and
to highlight some issues that appear in the presence of jumps. For
examples of processes satisfying the conditions below, see Section 4.1
below or the
articles~\autocite{ollivier_ricci_2009,joulin_curvature_2010,}.

Consider a discrete-time Markov Process $(X_t)_{t \in \mathbb{N}}$
taking values in $\R$ with $X_0 = x_0 \in \R$. Fix a measurable
function $f: \mathbb{N} \times \R \to \R$ and define the associated
two-parameter semigroup $P_{t,u} f$ by
\begin{equation*}
	P_{t,u} f(x) = \E[f(u, X_u) | X_t = x], \quad 0 \leq t \leq u; u, t \in \mathbb{N}.
\end{equation*}
Let
\begin{equation*}
S_t = \sum_{u=1}^{t} f(u, X_u), \quad t \geq 1.
\end{equation*}

\begin{proposition}\label{prop:discrete}
  Assume that there exist positive processes $C_t, \kappa_t, \sigma_t$
  with $\kappa_t < 1$ such that for each $t$, $C_t$, $\kappa_t$ and
  $\sigma_t$ are bounded, $\F_{t-1}$-measurable random variables and
  such that
\begin{align*}
\abs{X_t - X_{t-1}} &\leq C_t, \quad t \geq 1, \\
\abs{P_{s,t} f(x) - P_{s,t} f(y)} &\leq \sigma_s\, (1-\kappa_s)^{t-s}\,
  \abs{x-y}, \quad x,y \in \R^n, 0 \leq s \leq t.
\end{align*}

Then for all $T, R, a > 0$ we have
\begin{align*}
  \PP \left(S_T - \E S_T \geq R \;; \sum_{t=1}^T \frac{\sigma_t^2 \, C_t^2}{\kappa_t^2} \leq a^2\right)
  \leq \exp\left(-\frac{R^2}{8 a^2}\right).
\end{align*}
\end{proposition}

\begin{proof}
Fix $T > 0$ and define a martingale $M^T$ by
\begin{align*}
  M^T_t = \E^{\F_t} S_T = \sum_{u=1}^t \E^{\F_t} f(u, X_u)
\end{align*}
so that
\begin{align*}
  M^T_t - M^T_{t-1} 
	&= \sum_{u=1}^T \left(\E^{\F_t} f(u, X_u) - \E^{\F_{t-1}} f(u, X_u)\right) 
\end{align*}

We have by the Markov property, adaptedness of $X_t$ and our assumptions
\begin{align*}
  \lefteqn{\E^{\F_t} f(u, X_u) - \E^{\F_{t-1}} f(u, X_u)}\quad \\
  &= P_{t,u} f(X_t) - \E^{\F_{t-1}} P_t,u f(X_t) \\
  &= P_{t,u} f(X_t) - P_{t,u} f(X_{t-1}) - \E^{\F_{t-1}}\left[P_{t,u}
    f(X_t) - P_{t,u} f(X_{t-1})\right] \\
  &\leq \sigma_t (1-\kappa_t)^{u-t} \left(\Abs{X_t - X_{t-1}} +
    \E^{F_{t-1}}\Abs{X_t - X_{t-1}}\right) \\
  &\leq 2 \sigma_t (1-\kappa_t)^{u-t} C_t
\end{align*}
This shows that the increments of the martingale $M^T_t$ are uniformly bounded by a
predictable process independent of $T$:
\begin{align*}
  M^T_t - M^T_{t-1} &\leq 2\,C_t\,\sigma_t \sum_{u=t}^T  (1-\kappa_t)^{u-t}
  \leq \frac{2\,C_t\,\sigma_t}{\kappa_t}.
\end{align*}

An extension of the classical Azuma-Hoeffding inequality (see
for example Theorem 3.4 in~\autocite{bercu_concentration_2015,}) states
that for any square-integrable martingale $M$ with $M_0 = 0$ and such
that $M_t - M_{t-1} \leq D_t$ for bounded, $\F_{t-1}$-measurable
random variables $D_t$
we have the inequality
\begin{align*}
  \PP(M_T \geq x \;; \sum_{t=1}^T D_t^2 \leq y) \leq \exp\left(-\frac{x^2}{2y}\right).
\end{align*}

Since we assumed $X_0$ to be deterministic, we have $M^T_T - M^T_0 =
S_T - \E^{\F_0} S_T = S_T - \E S_T$. By applying the preceding
martingale inequality to $M^T$ we get finally
\begin{align*}
  \PP \left(S_T - \E S_T \geq R \;; \sum_{t=1}^T \frac{\sigma_t^2\, C_t^2}{\kappa_t^2} \leq a^2\right)
  \leq \exp\left(-\frac{R^2}{8 a^2}\right).
\end{align*}

\end{proof}



\section{Martingale and concentration inequalities}\label{sec:main}
\subsection{Square integrable processes}\label{sec:general}
Consider a filtered probability space
$(\Omega, \mathcal{F}, \PP, (\mathcal{F}_t)_{t\geq 0})$ satisfying the usual
conditions from the general theory of semimartingales, meaning that $\F$ is
$\PP$-complete, $\F_0$ contains all $\PP$-null sets in $\F$ and $\F_t$ is
right-continuous. In this section $(X_t)_{t\geq0}$ will denote a real-valued
stochastic process adapted to $\F_t$, square-integrable in the sense
that $\E X_t^2 < \infty$ for all $t \geq 0$.

Define an adapted continuous finite-variation process $(S_t)_{t\geq 0}$ by
\begin{equation*}
  S_t = \int_0^t X_u\, du.
\end{equation*}
Fix $T > 0$ and define a martingale $(M^T_t)_{t\geq 0}$ by
\begin{equation*}
  M^T_t = \E^{\F_t} S_T.
\end{equation*}
By the boundedness and adaptedness assumptions on $X$, $M^T$ is a square integrable
martingale (by Doob's maximal inequality) which we can and will choose to be right-continuous with left limits so that the predictable quadratic variation $\pQV{M^T}$ and the jumps $(\Delta
M^T_t)_{t\geq0}$ are well-defined.


As illustrated in Section~\ref{sec:direct}, our goal is to derive
concentration inequalities for $S_T$ from concentration inequalities
for the martingale $M^T$ by using the relation $S_T = M^T_T$. In this
subsection we take the first step towards that goal by characterizing
$M^T$ (and thus $S_T$) based on properties of the underlying process
$X$, using a family of auxiliary martingales. The next two subsections
will give explicit expressions for the auxiliary martingales in a
martingale and Markov process context. The final subsection will
complete the approach by recalling known concentration inequalities
for martingales in terms of their (predictable) quadratic variations
and jumps.

Define the family of auxiliary martingales
$(Z^u)_{ u\geq0}$ by
\begin{equation*}
  Z^u_t = \E^{\F_t} X_u,
\end{equation*}
which will be chosen right-continuous with left limits. Each $Z^u$ is square
integrable so that the predictable quadratic covariation $\pQV{Z^u, Z^v}$ is
well-defined.

Formally the next result is just a consequence of the (bi)linearity of
the integral and (predictable) quadratic variation. The proof shows
that the formal calculation is justified under our assumption that $X$
is square integrable. The main interest of the result lies in the fact
that we can often find explicit expressions for $\pQV{Z^u, Z^v}$ and
$\Delta Z^u$, as we will see in the next two sections.

\begin{theorem}\label{lemma:1}\label{thm:main}\label{theorem:main}
For any $T > 0$ we have 
\begin{align}
M^T_t &= \int_0^T Z^u_t du, \label{eq:M} \\
\Delta M_t^T &= \int_0^T \Delta Z^u_t du, \label{eq:Delta} \\
\QV{M^T}_t &= \int_0^T \int_0^T \QV{Z^u, Z^v}_t \,du\, dv, \label{eq:QV} \\
\pQV{M^T}_t &= \int_0^T \int_0^T \pQV{Z^u, Z^v}_t \,du\, dv.
\label{eq:pQV}
\end{align}
\end{theorem}

\begin{proof}
  Since the theorem only involves values of $X_t$ and $Z^u_t$ for
  $u,t \leq T$, we can assume without loss of generality that
  $Z^u_t = Z^{u\wedge T}_{t\wedge T}$ and $X_t = X_{t\wedge T}$. We
  proceed with some preliminary estimates.  First, the family
  $(Z^u_\tau)_{u,\tau}$, $u \geq 0$, $\tau$ stopping time, is
  uniformly bounded in $L^2$ since, by optional sampling and Jensens's inequality,
  \begin{equation*}
    \E (Z^u_\tau)^2 = \E (Z^{u\wedge T}_{\tau\wedge T})^2 = \E
    (\E^{\F_{\tau \wedge T}} \E^{\F_T} X_{u\wedge T})^2 \leq \E X_{u \wedge T}^2 \leq
    \sup_{0\leq t\leq T} \E X_t^2 < \infty.
  \end{equation*}
  It follows that $(\QV{Z^u, Z^v}_\tau)_{u,v,\tau}$ and
  $(\pQV{Z^u, Z^v}_\tau)_{u,v,\tau}$ are uniformly bounded in $L^1$
  since, by elementary properties of the quadratic variation,
  Cauchy-Schwartz and uniform integrability of the martingales $Z^u$,
  \begin{equation*}
    \E \Abs{\QV{Z^u, Z^v}_\tau} \leq \E (\QV{Z^u}_\tau \, \QV{Z^v}_\tau)^{1/2} \leq (\E (Z^u_\tau)^2)^{1/2} (\E (Z^v_\tau)^2)^{1/2}.
  \end{equation*}
  Furthermore the local martingale $Z^u Z^v - \QV{Z^u, Z^v}$ is actually a
  uniformly integrable martingale since
  \begin{equation*}
    \E \Abs{Z^u_\infty Z^v_\infty - \QV{Z^u, Z^v}_\infty}
    = \E \Abs{Z^u_T Z^v_T - \QV{Z^u, Z^v}_T}
    < \infty.
  \end{equation*}
  The same holds for $Z^u Z^v -
  \pQV{Z^u, Z^v}$ by an identical argument.
  
We proceed to show \eqref{eq:M}, i.e. for fixed $T > 0$ and all $t
\geq 0$
\begin{equation}\label{eq:M_explicitt}
\E^{\F_t} \int_0^T X_u  \, du = \int_0^T Z^u_t \, du.
\end{equation}
Since $(Z^u_t)_{u\geq0}$ is uniformly bounded in $L^1$ we get by Fubini's theorem
that for arbitrary $G \in \F_t$
\begin{equation*}
  \E \left[\int_0^T Z^u_t
    \, du ; G\right] = \int_0^T \E[Z^u_t; G] du = \int_0^T \E[X_u; G]
  du = \E \left[ \int_0^T X_u  \, du ; G \right]
\end{equation*}
which is equivalent to~\eqref{eq:M_explicitt}.

To show~\eqref{eq:Delta} note that since $(Z^u_t)_{t\geq0}$ is
uniformly integrable we have $Z^u_{t-1/n} \to Z^u_{t-}$ in
$L^1$. Since $(Z^u_t)_{u,t\geq0}$ is uniformly bounded in $L^1$ we can
use Fubini's theorem and dominated convergence to show that
\begin{equation*}
  M^T_{t-1/n} = \int_0^T Z^u_{t-1/n} \,du \to \int_0^T Z^u_{t-}\,du \text{ in } L^1.
\end{equation*}
Since the almost sure limit of the left-hand side is $M^T_{t-}$ by
definition and the almost sure and $L^1$ limits coincide when they
both exists, this proves~\eqref{eq:Delta}.

To show \eqref{eq:QV}, we are going to use the
characterisation of $\QV{M^T}$ as the unique adapted and \emph{càdlàg}
process $A$ with paths of finite variation on compacts such
that $(M^T)^2 - A$ is a local martingale and
$\Delta A = (\Delta M^T)^2, A_0 = (M^T_0)^2$.

Let
\begin{equation*}
  A_t = \int_0^T \int_0^T \QV{Z^u, Z^v}_t \, du \, dv.
\end{equation*}
Clearly $A_0 = (M^T_0)^2$ since $\QV{Z^u, Z^v}_0 = Z^u_0 Z^v_0$. Since sums and limits preserve measurability, $A$ is adapted since all the
integrands $\QV{Z^u, Z^v}$ are. By polarization
\begin{equation*}
  A_t = \frac14 \int_0^T \int_0^T \QV{Z^u + Z^v}_t \, du \, dv - \frac14
  \int_0^T \int_0^T \QV{Z^u - Z^v}_t \, du \, dv
\end{equation*}
so that the paths of $A_t$ can be written as a difference between two
increasing functions and are thus of finite variation on compacts.
Futhermore, since
$0 \leq \QV{Z^u \pm Z^v}_t \leq \QV{Z^u \pm Z^v}_{u \vee v}$ for all
$t \geq 0$ and $\QV{Z^u \pm Z^v}_t$ is increasing in $t$, pathwise
left limits and right continuity follow from the monotone convergence
theorem applied to the preceding decomposition.

Since $\QV{Z^u, Z^v}_t$ is bounded in $L^1$, uniformly in $u, v, t$,
and uniformly integrable in $t$, it follows as in the proof
of~\eqref{eq:Delta} that
\begin{equation*}
  A_{t-} = \int_0^T \int_0^T \QV{Z^u, Z^v}_{t-}\,du\,dv
\end{equation*}
so that
\begin{equation*}
  \Delta A_t = \int_0^T \int_0^T \Delta\QV{Z^u, Z^v}_{t}\,du\,dv
  = \int_0^T \Delta Z^u_t \,du \int_0^T \Delta Z^v_t \,dv
  = \left(\Delta M^T_t\right)^2.
\end{equation*}
To show the martingale part of the characterization, recall that an
adapted càdlàg process $Y$ is a uniformly integrable martingale if for
all stopping times $\tau$, $\E\Abs{Y}_\tau < \infty$ and $\E Y_\tau = 0$.
Now, for any
stopping time $\tau$, by our preliminary estimates
above we get the integrability so that we can use Fubini's theorem
and optional
stopping to obtain
\begin{equation*}
  \E \left[\left(M^T_\tau\right)^2 - A_\tau\right] = \int_0^T \int_0^T \E\left(Z^u_\tau
  Z^v_\tau - \QV{Z^u, Z^v}_\tau\right)\,du\,dv = 0
\end{equation*}
which finishes the proof that $\QV{M^T}_t = A_t$.

It remains to identify $\pQV{M^T}$ as the compensator of $\QV{M^T}$,
i.e.\ the unique finite variation predictable process $\tilde{A}$
such that $A - \tilde{A}$ is a local martingale. Let
\begin{equation*}
  \tilde{A}_t = \int_0^T \int_0^T \pQV{Z^u, Z^v}_t\,du\,dv.
\end{equation*}
Then $\tilde{A}$ is predictable since measurability is preserved by
sums and limits, and thus by integrals. Using polarization, we
see that $\tilde{A}$ is the difference between two increasing
processes and therefore of finite variation. Finally, for all stopping
times $\tau$, we again get from our preliminary estimates that
$\E\abs{A_\tau - \tilde{A}_\tau} < \infty$ and by optional stopping
\begin{equation*}
  \E [A_\tau - \tilde{A}_\tau] = \int_0^T \int_0^T \E\left(\QV{Z^u,
      Z^v}_\tau - \pQV{Z^u, Z^v}_\tau\right)\,du\,dv = 0
\end{equation*}
which shows that $A - \tilde{A}$ is a martingale and thereby concludes
the proof of the theorem.

\end{proof}

\begin{proposition}
    If there exist real-valued processes $\sigma_t$, $J_t$ and a constant
  $\kappa \geq 0$ such that for all $ 0 \leq t \leq u \leq T$
  \begin{align*}
    d\pQV{Z^u}_t &\leq \sigma_t^2 e^{-2\kappa (u-t)} dt, \\
    \abs{\Delta Z^u_t} & \leq \abs{\Delta J_t} e^{-\kappa (u-t)}
  \end{align*}
  then
  \begin{align*}
    	\pQV{M^T}_t &
    \leq \int_0^t \frac{\sigma_s^2}{\kappa^2} \left(1 - e^{-\kappa (T-s)}\right)^2 ds, \\
    \abs{\Delta M^T_t} & \leq \frac{\abs{\Delta J_t}}{\kappa} \left(1 - e^{-\kappa (T-t)}\right)
  \end{align*}
  where the case $\kappa = 0$ is to be understood in the sense of the limit as
  $\kappa \to 0$.
\end{proposition}

\begin{proof}
  The second inequality is immediate from Theorem~\ref{lemma:1} and the observation that
  $\Delta Z^u_t = 0$ for $t > u$. We now proceed to prove the first one.
  For $t \leq u\wedge v$ we have
\begin{equation*}
	\pQV{Z^u}_t d\pQV{Z^v}_t \leq e^{-2\kappa u} \int_0^t \sigma^2_s e^{2\kappa
          s} ds \, e^{-2\kappa v} \sigma^2_t
        e^{2\kappa t} = \frac12 d\left(\int_0^t
          \sigma_s^2 e^{-\kappa (u-s)} e^{-\kappa (v-s)} ds \right)^2
\end{equation*}
which is symmetric in $u$ and $v$.
Using Cauchy-Schwarz for the predictable quadratic variation, the fact that $Z^u_t$ is constant for $t \geq u$ and integration by parts together with the previous inequality we get
\begin{align*}
	\pQV{Z^u, Z^v}_t  &\leq ( \pQV{Z^u} \pQV{Z^v})_{t\wedge u \wedge v}^{1/2} =
\left(\int_0^{t\wedge u\wedge v} \pQV{Z^u}_s d\pQV{Z^v}_s + \int_0^{t\wedge u\wedge
                            v} \pQV{Z^v}_s d\pQV{Z^u}_s\right)^{1/2} \\
&\leq \int_0^t \one_{\{s\leq u\wedge v\}} \sigma_s^2 e^{-\kappa(u-s)} e^{-\kappa(v-s)} ds.
\end{align*}
Therefore by Fubini, for any $0 \leq t \leq T$
\begin{align*}
	\pQV{M^T}_t
  &= \int_0^T \int_0^T \pQV{Z^u, Z^v}_t du dv \\
  & \leq \int_0^T \int_0^T \int_0^t \one_{\{s\leq u\wedge v\}} \sigma_s^2
    e^{-\kappa(u-s)} e^{-\kappa(v-s)} ds \, du \, dv \\
    & = \int_0^t \sigma_s^2 \left(\int_s^T e^{-\kappa (u-s)} du\right)^2 ds.
\end{align*}
which is the result.
\end{proof}

\subsection{Martingales}\label{sec:general:martingales}
\begin{proposition}\label{prop:Martingale}
  If $X$ is a square integrable real-valued martingale then  
\begin{align*}
  dM^T_t &= (T-t) \, dX_t \\
  \Delta M^T_t &= (T-t) \, \Delta X_t \\
  d\QV{M^T}_t &= (T-t)^2 \, d\QV{X}_t \\
  d\pQV{M^T}_t &= (T-t)^2 \, d\pQV{X}_t
\end{align*}
\end{proposition}

\begin{proof}
Since
\begin{equation*}
  Z^u_t = \E^{\F_t} X_u = X_{t \wedge u}
\end{equation*}
we have by \eqref{eq:M}
\begin{equation*}
  M^T_t = \int_0^T X_{t\wedge u} \,du = \int_0^t X_u\, du + (T-t) X_t.
\end{equation*}
Using integration by parts
\begin{equation*}
  dM^T_t = X_t \,dt + (-X_t \, dt + (T-t) \, dX_t) = (T-t) dX_t
\end{equation*}
and the remaining equalities now follow directly from stochastic calculus.
\end{proof}

\begin{remark}\label{remark:martT2}
  From integration by parts (and similarly for $\pQV{M^T}$)
  \begin{equation*}
    \QV{M^T}_T = 2\int_0^T (T-t) \QV{X}_t\,dt.
  \end{equation*}
\end{remark}

From the fact that the Doléans-Dade exponential is a positive local martingale and therefore a supermartingale we immediately get the following corollary.
\begin{corollary}
  If the martingale $X$ is continuous and $X_0 = x \in \R$ then for all
  $\lambda \in \mathbb{C}$ and $T > 0$
  \begin{equation*}
    \E \exp\left(\lambda (S_T - \E S_T) - \lambda^2 \int_0^T (T-u)^2 \,d\pQV{X}_u\right) \leq 1.
  \end{equation*}
\end{corollary}

\begin{remark}[Central limit theorem]\label{remark:CLT}
  From the Doléans-Dade exponential it is also possible to derive a
  central limit theorem for a suitably normalized family of random
  variables $G_T = \frac1{\sqrt{\E\QV{M^T}_T}} \int_0^T X_t\,dt$. A
  thorough investigation is beyond the scope of this work, but an
  outline of the argument goes as follows. Denote $\mathcal{E}$ the
  Doléans-Dade exponential and for two families of random variables
  $X_T, Y_T$ write $X_T \Psim Y_T$ if $X_T/Y_T \to 1$ in probability
  as $T \to \infty$. Now suppose that $\QV{M^T}_T \Psim \E\QV{M^T}_T$,
  which can for example follow from an ergodic theorem since
  $\QV{M^T}_T$ is usually an additive functional itself. Then for
  $\lambda \in \R$ we have for the characteristic function of $G_T$
  that
  \begin{align*}
    e^{i \lambda G_T} &= \exp\left(i\lambda \frac{M^T_T}{\sqrt{\E\QV{M^T}_T}} +
    \frac{\lambda^2}{2 \E\QV{M^T}_T} \QV{M^T}_T - \frac{\lambda^2}{2 \E\QV{M^T}_T}
    \QV{M^T}_T\right) \\ & = \mathcal{E}\left(\frac{i\lambda
    M^T_T}{\sqrt{\E\QV{M^T}_T}}\right) \exp\left(-\frac{\lambda^2 \QV{M^T}_T}{2
                        \E\QV{M^T}_T} \right)
    \Psim \mathcal{E}\left(\frac{i\lambda
    M^T_T}{\sqrt{\E\QV{M^T}_T}}\right) e^{-\frac{\lambda^2}{2}}.
  \end{align*}
  Now if the Doléans-Dade exponential (which can be negative since we
  have complex exponent) is a true martingale and the
  family $\mathcal{E}\left(\frac{i\lambda
      M^T_T}{\sqrt{\E\QV{M^T}_T}}\right)$ is uniformly integrable then for all
  $\lambda \in \R$
  \begin{equation*}
    \lim_{T\to\infty} \E e^{i \lambda G_T} = \lim_{T\to\infty} \E\, \mathcal{E}\left(\frac{i\lambda
        M^T_T}{\sqrt{\E\QV{M^T}_T}}\right) e^{-\frac{\lambda^2}{2}} = e^{-\frac{\lambda^2}{2}} 
  \end{equation*}
  so that $G_T$ converges in distribution to a standard normal random
  variable.
\end{remark}

\subsection{Markov Processes}\label{sec:general:markov}
Consider a continuous-time Markov process $(Y_t)_{t\geq0}$ with natural filtration
$(\F_t)_{t\geq 0}$, taking values in a Polish space $E$ and with trajectories that
are right-continuous with left limits. Denote $\B$ the set of Borel functions on
$\R^+ \times E$.

Fix a function $f \in \B$ such that $\sup_t \E f(t, Y_t)^2 < \infty$ so that we are in
the setting of Section~\ref{sec:general} with $X_t = f(t, Y_t)$, $S_T = \int_0^T
f(u, X_u) du$ and $M^T_t = \E^{\F_t} S_T$.
Define the two-parameter semigroup $(P_{t,u}f)_{0 \leq t \leq u}$ on $E$ by
\begin{equation*}
P_{t,u} f(y) = \E[f(u, Y_u) | Y_t = y].
\end{equation*}


Suppose that there is a set $\mathcal{D}(\Gamma) \subset \B \times \B$ and a map
$\Gamma: \mathcal{D}(\Gamma) \to \B$ such that for each $(f, g) \in
\mathcal{D}(\Gamma)$ we have, setting $F_t = f(t, Y_t), G_t = g(t, Y_t)$, that
\begin{equation}\label{eq:Gamma}
  \pQV{F, G}_t = 2 \int_0^t \Gamma(f, g)(s, Y_s)\,ds,\quad 0 \leq t \leq T.
\end{equation}

Usually, $\Gamma$ corresponds to (an extension of) the squared field
operator $\Gamma(f, g) = Lfg - fLg - gLf$, see
Remark~\ref{remark:gamma} below.

\begin{proposition}\label{prop:Markov}
  If $(P_{\cdot,u} f, P_{\cdot,v} f)_{0 \leq u,v \leq T} \in \mathcal{D}(\Gamma)$ we have
  for $0 \leq t \leq T$
\begin{align}
  d \pQV{M^{T}}_t &= 2 \int_t^T \int_t^T \Gamma(P_{t,u} f, P_{t,v} f)(t, Y_t) \, du
                      \, dv \, dt, \label{eq:pQVMarkov}\\
	\Delta M^{T}_t &= \int_t^T \left( P_{t,u} f(Y_t) - P_{t,u} f(Y_{t-}) \right) \, du. \label{eq:DeltaMarkov}
\end{align}
\end{proposition}
\begin{proof}
  For all $0 \leq t \leq u \leq T$ we have
  \begin{equation*}
  Z^u_t = \E^{\F_t} f(u, Y_u) = P_{t,u}f(Y_t).
\end{equation*}
  Since $f$ can depend on $u$ we can
  always consider $f(u, y) - P_{0,u}f(y)$ instead of $f$ and therefore without loss of
  generality suppose that $Z^u_0 = P_{0,u}f(Y_0) = 0$.
  By~\eqref{eq:Gamma} for $0 \leq t \leq u \wedge v \leq T$
  \begin{equation*}
    d\pQV{Z^u, Z^v}_t = 2 \Gamma(P_{t,u} f, P_{t,v} f)(t, Y_t) dt, \quad 0 \leq t \leq u \wedge v \leq T.
  \end{equation*}
  For $t \geq v \wedge u$ either $Z^u_t$ or $Z^v_t$ is constant so that
  \begin{equation*}
    d\pQV{Z^u, Z^v}_t = 0, \quad u \wedge v \leq t \leq T,
  \end{equation*}
  and thus
  \begin{equation*}
    \pQV{Z^u, Z^v}_t = \int_0^t \one_{\{s \leq u \wedge v\}} 2 \Gamma(P_{s,u} f, P_{s,v} f)(s,
    Y_s) ds.
  \end{equation*}
  By~\eqref{eq:pQV} and Fubini's theorem
  \begin{align*}
    \pQV{M^{T}}_t 
&= \int_0^T \int_0^T \int_0^t \one_{\{s \leq u \wedge v\}} 2 \Gamma(P_{s,u} f,
  P_{s,v} f)(s, Y_s) \,ds \,du \, dv \\
&= \int_0^t \int_s^T \int_s^T 2 \Gamma(P_{s,u} f,
  P_{s,v} f)(s, Y_s) \, du \, dv \, ds.
  \end{align*}
  This proves the first equality~\eqref{eq:pQVMarkov}. Equality~\eqref{eq:DeltaMarkov}
  follows directly from~\eqref{eq:Delta} together with the observation that $Z^u_t$
  is constant
  for $t \geq u$ and the fact
  that $P_{t,u} f$ is continuous in $t$.
\end{proof}

\begin{remark}\label{remark:gamma}
When $Y$ is a Markov process with infinitesimal generator
$(L, \mathcal{D}(L) \subset \B)$ then $\Gamma$ in~\eqref{eq:pQVMarkov} corresponds to
the usual squared field operator whenever the latter is
well-defined,
\begin{equation*}
  \Gamma(f, g) = \frac12(L fg - f L g - g L f), \quad f
  \in \mathcal{D}(L), g \in \mathcal{D}(L), fg \in \mathcal{D}(L).
\end{equation*}
Indeed, suppose that for $f \in \mathcal{D}(L)$
\begin{equation}\label{eq:MP}
	f(t, Y_t) - f(0, Y_0) - \int_0^t(\partial_s f + L f(s, \cdot))(s, Y_s) \,ds
\end{equation}
is a local martingale.
As before we can assume $P_{0,u} f(Y_0) = 0$. Now if $P_{t,u}f$, $P_{t,v}f$ and
their product $P_{t,u}f
P_{t,v} f$ are in $\mathcal{D}(L)$ then
\begin{equation}\label{eq:MPPt}
  P_{t,u} f(Y_t)P_{t,v} f(Y_t) - \int_0^t(\partial_s (P_{s,u}f P_{s,v}f) +
  L(P_{s,u}f P_{s,v} f))(s, Y_s) \,ds
\end{equation}
is a local martingale. Since we assumed $P_{t, u} f$ to be in $\mathcal{D}(L)$ it
solves the Kolmogorov backward equation
  \begin{equation*}
    \partial_t P_{t,u}f(y) = -L P_{t, u}(t, y), \quad y \in E, 0 \leq t \leq u
  \end{equation*}
and the same holds true for $P_{t,v} f$. Thus
\begin{equation*}
  \partial_t (P_{t,u} f P_{t,v} f) = P_{t,u} f \partial_t P_{t,v} f + P_{t,v} f
  \partial_t P_{t,u} f = - P_{t,u} f L P_{t,v} f - P_{t,v} f
  L P_{t,u} f.
\end{equation*}
Substituting this into the integral in~\eqref{eq:MPPt} shows that indeed
\begin{equation*}
  P_{t,u} f(Y_t)P_{t,v} f(Y_t) - 2 \int_0^t \frac12(L(P_{s,u}f P_{s,v} f) - P_{s,u} f
  L P_{s,v} f - P_{s,v} f L P_{s,u}f)(s, Y_s) \, ds
\end{equation*}
is a local martingale.
\end{remark}

\begin{corollary}
  If $Y$ has continuous trajectories and $Y_0$ is constant then we have the
  following inequality for all $\lambda \in \mathbb{C}$ and $T > 0$ fixed:
  \begin{equation*}
    \E \exp\left[\lambda \left(S_T f - \E S_T f\right) - \lambda^2 \int_0^T
      \int_t^T \int_t^T \Gamma(P_{t,u} f, P_{t,v} f) (t, Y_t) \,du \, dv\, dt\right] \leq 1.
  \end{equation*}
\end{corollary}
\begin{proof}
  This follows directly from the Chernoff bound and the fact that the
  Doléans-Dade exponential is a supermartingale.
\end{proof}

\begin{remark}[Central limit theorem] The considerations in
  Remark~\ref{remark:CLT} for deriving a central limit theorem for
  additive functionals of martingales apply to continuous
  Markov processes as well.
\end{remark}

\subsection{Martingale inequalities}\label{subsec:martineq}
Let
\begin{equation*}
  S_T = \int_0^T X_u\,du
\end{equation*}
for some square integrable càdlàg process $X$ and
$M^T_t = \E^{\F_t}S_T$ as in the previous sections. Our key
observation is that $S_T - \E^{\F_0} S_T = M^T_T -
M^T_0$. Concentration inequalities for $S_T$ then follow from
concentration inequalities for martingales. The goal of this section
is to show how to pass from $\E^{\F_0} S_T$ to $\E S_T$ and to recall
some concentration inequalities for martingales.

For a real-valued random variable $Y$, denote $\Psi_Y(\lambda)$ the logarithm of the moment-generating
function of $Y$ and $\Psi_Y^*(x)$ its associated Cramér transform:
\begin{align*}
  \Psi_Y(\lambda) &= \log \E e^{\lambda Y}, \quad \lambda \in \R, \\
  \Psi_Y^*(x) &= \sup_{\lambda \in \R}(\lambda x - \Psi_Y(\lambda)),
                \quad x \in \R.
\end{align*}

Denote $\Lambda(\lambda)$ the logarithm of the moment-generating
function of the centered random variable $\E^{\F_0} S_T - \E S_T$ and $I$ its domain:
\begin{align*}
  \Lambda(\lambda) &= \Psi_{\E^{\F_0} S_T - \E S_T} = \log
                     \E\left[\exp \lambda \left(\E^{\F_0} S_T - 
                     \E S_T \right)\right], \\
  I &= \{ \lambda \in \R: \Lambda(\lambda) < \infty \}.
\end{align*}
In particular if $X_0 = x \in \R$ is a deterministic constant then $\E^{\F_0} S_T = \E S_T$ and $\Lambda(\lambda) = 0$.

Following~\autocite{dzhaparidze_bernstein-type_2001,}, define
\begin{align*}
  \varphi(x) &= e^x - 1 - x, \\
  \varphi_a(x) &= \varphi(ax)/a^2, \quad a \geq 0, \\
  H^a_t 
    &= \sum_{s\leq t} (\Delta M^T_s)^2 \one_{\{\abs{\Delta M^T_s} > a\}} +
    \pQV{M^T}_t, \quad a, t \geq 0
\end{align*}
where for $a=0$ we set $\varphi_0(x) = x^2/2$ and we have $H^0_t = \sum_{s\leq t}\left(\Delta
M^T_s\right)^2 + \pQV{M^T}_t$.

The next lemma allows us to extend our framework from processes with
initial measures concentrated on a single point to more general
classes of initial measures when we can control $\Lambda$.
  
\begin{lemma}\label{lemma:MGF}
  For $a \geq 0, \lambda \in I$
  \begin{equation*}
      \E \exp\Bigl(\lambda(S_T - \E S_T) - \varphi_a(\abs{\lambda}) H^a_T - \Lambda(\lambda))\Bigr)
    \leq 1.
  \end{equation*}
\end{lemma}

\begin{proof}
  In~\autocite{dzhaparidze_bernstein-type_2001,} Corollary 3.1 it is shown that for
  any square integrable martingale $M$ and for all $a \geq 0, \lambda \geq 0$ the
  process
  \begin{equation*}
    \exp\left(\lambda M_t - \varphi_a(\abs{\lambda})H_t^a\right)
  \end{equation*}
  is a supermartingale. Applying this to $M^T$ and $-M^T$ together with the
  supermartingale property yields that for $a \geq 0, \lambda \in \R$
  \begin{equation*}
    \E^{\F_0} \exp\biggl(\lambda (M^T_t - M^T_0) -
      \varphi_a(\abs{\lambda}) H_t^a\biggr) \leq 1.
  \end{equation*}
  By definition we have furthermore that for $\lambda \in I$
  \begin{equation*}
    \E \exp(\lambda(M^T_0 - \E M^T_0)) = \exp \Lambda(\lambda).
  \end{equation*}
  Therefore for $\lambda \in I$ and all $t \in [0, T]$
  \begin{align*}
\lefteqn{    \E \exp\left(\lambda(M^T_t - \E M^T_t) - \varphi_a(\abs{\lambda})H^a_t -
    \Lambda(\lambda)\right) }\quad \\
    &= \E \left\{ \E^{\F_0} \left[\exp\Bigl(\lambda(M^T_t - M^T_0) -
      \varphi_a(\abs{\lambda}) H^a_t\Bigr) \right]
      \exp\Bigl(\lambda(M^T_0 - \E M^T_0) - \Lambda(\lambda)\Bigr)
      \right\} \\
    &\leq \E \exp\left(\lambda(M^T_0 - \E M^T_0) - \Lambda(\lambda)\right)
    = 1.
  \end{align*}
  We conclude by taking the inequality at $t = T$ and noting that $M^T_T$ = $S_T, \E M^T_T = \E S_T$.
\end{proof}

In the previous sections, we saw how to estimate the quantities $\Delta
M^T$ and $\pQV{M^T}$, and thus $H^a$, for different classes of
processes $X$. We will now recall some martingale inequalities
involving $H^a$, which then lead directly to inequalities for $S_T -
\E S_T$.

From Markov's inequality applied to $e^{\lambda Y}$ we immediately get Chernoff's
inequality
\begin{equation*}
  \PP\{Y \geq x\} \leq \exp(-\Psi^*_Y(x)).
\end{equation*}
By combining this with Lemma~\ref{lemma:MGF} and bounds on $\Lambda$ we can
immediately deduce the following Hoeffding, Bennett and Bernstein-type
inequalities. The approach is classical and we follow~\autocite{boucheron_concentration_2013}.

\begin{corollary}\label{corr:subGaussian}
  If $\Lambda(\lambda) \leq \frac{\lambda^2}{2} \rho^2$ for some $\rho \geq 0$ then
  \begin{equation*}
    \PP\left(S_T - \E S_T \geq R\, ; H^0_T \leq \sigma^2\right) \leq \exp\left(-\frac{R^2}{2 (\rho^2 + \sigma^2)}\right).
  \end{equation*}
\end{corollary}
\begin{proof}
  On the set $\{ H^0_T \leq \sigma^2 \}$, using
  $\varphi_0(\lambda) = \frac{\lambda^2}{2}$, $\Psi_{S_T - \E S_T}$ is upper bounded
  by the logarithmic MGF of a centered Gaussian random variable with variance
  $\rho^2 + \sigma^2$:
  $\Psi_{S_T - \E S_T}(\lambda) \leq \frac{(\rho^2 + \sigma^2)\lambda^2}{2}$. This
  implies that $\Psi^*_{S_T - \E S_T}$ is lower bounded by the corresponding Cramér
  transform, $\Psi^*_{S_T - \E S_T}(x) \geq \frac{x^2}{2 (\rho^2 + \sigma^2)}$, and
  the result follows immediately from Chernoff's inequality.
\end{proof}

\begin{corollary}\label{corr:subPoisson}
  If $\Lambda(\lambda) \leq \nu \varphi_a(\lambda)$ for some $a, \nu \geq 0$ then
  \begin{equation*}
    \PP\left(S_T - \E S_T \geq R \,; H^a_T \leq \mu\right) \leq \exp\left(-\frac{\mu + \nu}{a^2}h\left(\frac{a
      R}{\mu + \nu}\right)\right)
\end{equation*}
with
\begin{equation*}
  h(x) = (1+x)\log(1+x) - x, \quad x \geq -1.
\end{equation*}
\end{corollary}
\begin{proof}
  On the set $\{ H^a_T \leq \mu \}$, $\Psi_{S_T - \E S_T}(\cdot/a)$ is upper
  bounded by the logarithmic MGF of a centered Poisson random variable with parameter
  $\frac{\mu + \nu}{a^2}$:
  $\Psi_{S_T - \E S_T}(\lambda / a) \leq \frac{(\mu +
    \nu)\varphi(\lambda)}{a^2}$. This implies that
  \begin{equation*}
\Psi^*_{S_T - \E S_T}(ax) =
  \sup_{\lambda \geq 0}(\lambda a x - \Psi_{S_T - \E S_T} (\lambda)) = \sup_{\lambda \geq 0}(\lambda x - \Psi_{S_T - \E S_T} (\lambda/a))
\end{equation*}
is
  lower bounded by the corresponding Cramér transform,
  $\Psi^*_{S_T - \E S_T}(ax) \geq \frac{\mu + \nu}{a^2}h\left(\frac{a^2 x}{\mu +
      \nu}\right)$, and the result follows from Chernoff's inequality after
  rescaling by $a$.
\end{proof}

\begin{corollary}
    If $\Lambda(\lambda) \leq \frac{\lambda^2 \nu}{2(1-b \lambda)}$ for some $b, \nu
    \geq 0$ and all $\lambda < 1/b$ then
    \begin{equation*}
       \PP\left(S_T - \E S_T \geq R\,; H^0_T \leq \mu\right) \leq \exp\left(-\frac{\mu + \nu}{b^2}h_1\left(\frac{b
      R}{\mu + \nu}\right)\right)
\end{equation*}
with
\begin{equation*}
  h_1(x) = 1 + x - \sqrt{1 + 2x}.
\end{equation*}
\end{corollary}
\begin{proof}
  On the set $\{ H^0_T \leq \mu \}$ using
  $\varphi_0(\lambda) = \frac{\lambda^2}{2} \leq \frac{\lambda^2}{2(1-\lambda b)}$,
  $\Psi_{S_T - \E S_T}$ is upper bounded by the (rescaled) logarithmic MGF of a
  sub-Gamma random variable (using the terminology
  of~\autocite{boucheron_concentration_2013}) with parameter $(\mu+\nu, b)$:
  $\Psi_{S_T - \E S_T}(\lambda) \leq \frac{(\mu +
    \nu)\lambda^2}{2(1-b\lambda)}$. This implies that $\Psi^*_{S_T - \E S_T}$ is
  lower bounded by the corresponding Cramér transform,
  $\Psi^*_{S_T - \E S_T}(x) \geq \frac{\mu + \nu}{b^2}h_1\left(\frac{b x}{\mu +
      \nu}\right)$, and the result follows as before from Chernoff's inequality.
\end{proof}

Going beyond the Chernoff inequality, we have for example the following result which
follows directly from Lemma~\ref{lemma:MGF} and an inequality on self-normalized
processes in~\autocite{de_la_pena_exponential_2009} Theorem 2.1.

\begin{corollary}\label{corr:selfnorm}
  If $\Lambda(\lambda) \leq \frac{\lambda^2}{2} \rho^2$ for some $\rho \geq 0$ and
  all $\lambda \in \R$ then
  \begin{equation*}
    \PP\left(\frac{\Abs{S_T - \E S_T}}{\sqrt{\frac32 (H^0_T + 
\E H^0_T + 2\rho^2)}} \geq R\right) \leq
    \min\{2^{1/3}, (2/3)^{2/3} R^{-2/3}\} \exp\left(-\frac{R^2}{2}\right).
  \end{equation*}
\end{corollary}
\begin{proof}
  By Theorem 2.1 in~\autocite{de_la_pena_exponential_2009}, for a pair
  of random variables $(A, B)$ with $B > 0$ satisfying
  \begin{equation*}
    \E\left[\exp\left(\lambda A - \frac{\lambda^2}{2}
        B^2\right)\right] \leq 1, \quad \lambda \in \R
  \end{equation*}
  and $\E B^2 = \E A^2 < \infty$ we have
  \begin{equation}\label{eq:selfnormAB}
    \PP\left(\frac{\abs{A}}{\sqrt{\frac32 (B^2 + \E[A^2])}} \geq
      R\right)
    \leq \min\{2^{1/3}, (2/3)^{2/3} R^{-2/3}\} e^{-R^2/2}.
  \end{equation}
  The corollary is now a direct consequence of Lemma~\ref{lemma:MGF}
  (with $a = 0$).
\end{proof}

\begin{corollary}\label{corr:selfnormCD}
If $\Lambda(\lambda) \leq \frac{\lambda^2}{2} \rho^2$ for some $\rho \geq 0$ and
all $\lambda \in \R$ then
\begin{align*}
\lefteqn{\PP\left(\abs{S_T - \E S_T} \geq R\,; H^0_T + \E H^0_T + 2\rho^2 \leq C\abs{S_T - \E S_T} + D\right)} \quad \\
& \leq 2^{1/3} \left(1 \wedge \frac{CR + D}{3R^2}\right)^{1/3}
 \exp\left(-\frac{R^2}{3(CR+D)}\right), \quad C, D \geq 0.
\end{align*}
\end{corollary}

\begin{proof}
On the set $\{ \abs{S_T - \E S_T} \geq R \}$ we have by monotonicity of $x \mapsto \frac{x}{1+x}$ that
\begin{equation*}
\frac{\abs{S_T - \E S_T}}{\sqrt{\frac32 (C\abs{S_T - \E S_T} + D)}} \geq \frac{R}{\sqrt{\frac32 (CR + D)}}
\end{equation*}
and on $\{ H^0_T + \E H^0_T + 2\rho^2 \leq C\abs{S_T - \E S_T} + D \}$ we have
\begin{equation*}
\frac{\abs{S_T - \E S_T}}{\sqrt{\frac32 (H^0_T + \E H^0_T + 2\rho^2)}} \geq \frac{\abs{S_T - \E S_T}}{\sqrt{\frac32 (C\abs{S_T - \E S_T} + D)}}. 
\end{equation*}
Together with the previous corollary we get the result
\begin{align*}
\lefteqn{\PP\left(\abs{S_T - \E S_T} \geq R\,; H^0_T + \E H^0_T + 2\rho^2 \leq C\abs{S_T - \E S_T} + D\right)} \quad & \\
& \leq \PP\left(\frac{\abs{S_T - \E S_T}}{\sqrt{\frac32 (H^0_T + \E H^0_T + 2\rho^2)}} \geq
\frac{R}{\sqrt{\frac32 (CR + D)}} \right)\\
& \leq 2^{1/3} \left(1 \wedge \frac{CR + D}{3R^2}\right)^{1/3} \exp\left(-\frac{R^2}{3(CR + D)}\right).
\end{align*}

\end{proof}

\begin{corollary}\label{corr:selfnormCSTD}
If $\Lambda(\lambda) \leq \frac{\lambda^2}{2} \rho^2$ for some $\rho \geq 0$ and
all $\lambda \in \R$ then
\begin{align*}
\lefteqn{\PP\left(\abs{S_T - \E S_T} \geq R\,; H^0_T \leq C \abs{S_T} + D\right)} \quad \\
& \leq 2^{1/3} \left(1 \wedge \frac{CR + D'}{3R^2}\right)^{1/3}
                                                                                            \exp\left(-\frac{R^2}{3(CR+D')}\right), \quad C, D \geq 0 \\
  & \text{ with } D' = D + C \abs{\E S_T} + \E H^0_T + 2\rho^2
\end{align*}

\end{corollary}

\begin{proof}
On the set $\{H^0_T \leq C \abs{S_T} + D\}$ we have
\begin{equation*}
  H^0_T \leq C(\abs{S_T} - \abs{\E S_T}) + C \abs{\E S_T} + D \leq
  C \abs{S_T - \E S_T} + C \abs{\E S_T} + D
\end{equation*}
so that
\begin{equation*}
H^0_T + \E H^0_T + 2\rho^2 \leq C \abs{S_T - \E S_T} + D + C \abs{\E
  S_T} + \E H^0_T + 2\rho^2
\end{equation*}
and the result follows directly from the previous Corollary.
\end{proof}

\section{Applications}\label{sec:examples}
\subsection{Polyak-Ruppert Averages}\label{sec:polyakruppert}
In this section, we use the notation $\Delta X_t = X_t - X_{t-1}$ for
a discrete-time process $X$. The symbols $t, s, u, T$ will always denote
time variables taking values in $\mathbb{Z}^+$.

Consider the real-valued
process $X$ defined by the recursion
\begin{equation}\label{eq:Xrecursion}
  X_t = X_{t-1} - \alpha_t \,g(X_{t-1}, W_t), \quad X_0 = x
\end{equation}
with $x \in \R$, $(\alpha_t)_{t\in\mathbb{N}}$ a sequence in $\R$,
$W_t$ a sequence of independent identically distributed random
variables with common law $\mu$ such that $\mu$ has compact
support, and $g \colon \R \times \R \to \R$ is a function such that
\begin{equation}\label{eq:stochapproxeq1}
  0 < m(w) \leq \partial_x \, g(x, w) \leq M(w) < \infty, \quad x, w \in \R
\end{equation}
for some functions $m \colon \R \to \R$ and $M \colon \R \to \R$.

The recursion~\eqref{eq:Xrecursion} is an instance of the
Robbins-Monroe algorithm for finding a root of the function
$\bar{g}(x) = \int g(x, w) \mu(dw)$. In our case, the
assumption~\eqref{eq:stochapproxeq1} implies that $\bar{g}$ is the
derivative of some strongly convex function and that $\bar{g}$ is
Lipschitz continuous with Lipschitz constant
$\bar{M} = \int M(w) \mu(dw)$.

Under certain assumptions on $g$ and the sequence of step sizes
$\alpha_t$, it can be shown that $X_t$ converges almost surely to a
limit $x^*$ such that $\bar{g}(x^*) = 0$
(\autocite{robbins_stochastic_1951, jentzen_strong_2018}).
It was later shown that the
convergence rate of the algorithm could be improved by considering the
Polyak-Ruppert averages $\frac{1}{T}\sum_{t=0}^{T-1} X_t$ (\autocite{polyak_acceleration_1992,}).

Using the approach developed in Section~\ref{sec:discrete} we now show
how to obtain concentration inequalities for the Polyak-Ruppert
averages around their expected value in the sense that if $\alpha_t =
\lambda t^{-p}$ for $\lambda$ sufficiently small and $p < 1/2$ then (see
Corollary~\ref{corr:PolyakRuppert} below for the precise statement)
\begin{multline*}
  \PP\left(\frac{1}{T}\sum_{t=0}^{T-1} X_t - \E\left[\frac{1}{T}\sum_{t=0}^{T-1} X_t\right] \geq R \:;\: \sup_{t\leq T} \sup_w \Abs{g(X_t, w)}\leq
    G\right) \\
  \leq \exp\left(-\frac{(1-2p) \, \bar{m}^2 \, R^2 \, T}{32 \, G^2}\right).
\end{multline*}
In particular, the order in $T$, the dependence of the numerator on
$\partial_x g$ and of the
denominator on $g$ match the central
limit theorem in \autocite{fort_central_2015}.  By
\autocite{jentzen_strong_2018} we also have under some additional
conditions that for $t \geq 1$, $\E\Abs{X_t - x^*} \leq C t^{-p/2}$ for some
constant $C$, so that
\begin{align*}
  \E\left[\frac{1}{T}\sum_{t=0}^{T-1} X_t\right] - x^* &\leq \frac{X_0 - x^*}{T} + \frac{C}{T}
  \sum_{t=1}^{T-1} t^{-p/2} \leq \frac{X_0 - x^*}{T} + \frac{C}{T}
  \frac{(T-1)^{1-p/2}}{1-p/2} \\ & \leq \frac{X_0 - x^*}{T} + \frac{2C}{T^{p/2}}
\end{align*}
and $\E\left[\frac{1}{T}\sum_{t=0}^{T-1} X_t\right] \to x^*$ as $T \to \infty$.

\begin{lemma}\label{lemma:polyakCt}
  Let $g^*(x) = \sup_{w \in \supp(\mu)} g(x, w)$ and set $C_t =
  \Abs{\alpha_t \, g^*(X_{t-1})}$.
  The process $C_t$ defines a sequence of $\F_{t-1}$-measurable
  bounded random variables such that $\abs{\Delta X_t} \leq C_t$.
\end{lemma}
\begin{proof}
The $\F_{t-1}$-measurability is clear, as is the inequality
$\abs{\Delta X_t} = \Abs{\alpha_t\, g(X_{t-1}, W_t)} \leq \Abs{\alpha_t \,g^*(X_{t-1})}$.
Since $g$ is continuous, in order to show that $g^*(X_{t-1})$ is
bounded it is sufficient to show that $X_{t-1}$ is bounded. This
follows from a simple induction.
Indeed, for an arbitrary $s > 0$, suppose that $\abs{X_s} \leq R_s < \infty$ and
let $R_{s+1} = R_s + \alpha_{s+1} \sup_{\abs{x} \leq R_s} g^*(x)$. Then
$\abs{X_{s+1}} \leq R_s + \alpha_{s+1}\abs{g(X_s, W_{s+1})} \leq R_s + \alpha_{s+1}\sup_{\abs{x}
  \leq R_s} g^*(x) = R_{s+1} < \infty$. Since $R_0 = \abs{x} < \infty$
the conclusion follows by induction.
\end{proof}

\begin{lemma}
  For any Lipschitz function $f$ with Lipschitz constant
  $\llipnorm{f}$ and $x, y \in \R$, $0 \leq s \leq t$ we have
  \begin{equation*}
    \Abs{P_{s,t} f(x) - P_{s,t} f(y)} \leq \llipnorm{f}\, \Abs{x -
      y} \, \prod_{u=s+1}^t \sqrt{1-2 \alpha_u \bar{m} + \alpha^2_u \bar{M}}
  \end{equation*}
\end{lemma}
\begin{proof}
  Let $X^x_t, X^y_t$ be the values at time $t$ of the
  recursion~\eqref{eq:Xrecursion} started from $X_s = x$
  respectively $X_s = y$. Then by definition $P_{s,t}
  f(x) = \E f(X^x_t)$ so that
  \begin{equation*}
    \Abs{P_{s,t} f(x) - P_{s,t} f(y)}
    = \Abs{\E \left(f(X^x_t) - f(X^y_t)\right)} \leq \llipnorm{f} \sqrt{\E \left(X^x_t - X^y_t\right)^2}.
  \end{equation*}
  Now from summation by parts and the bounds on $\partial_x \,g$ that
  we assumed we get
  \begin{align*}
    \Delta (X^x-X^y)^2_t
    &= 2 (X^x_{t-1} - X^y_{t-1}) \Delta
      (X^x-X^y)_t + \left[\Delta (X^x-X^y)_t\right]^2 \\
    &\leq -2 \, \alpha_t \, m(W_t) (X^x_{t-1} - X^y_{t-1})^2 + \alpha_t^2\,
      M(W_t)^2 (X^x_{t-1} - X^y_{t-1})^2 \\
    &= -(2 \, \alpha_t \, m(W_t) - \alpha_t^2 \, M(W_t)^2) (X^x_{t-1} - X^y_{t-1})^2
  \end{align*}
  so that by developing the recursion we obtain
  \begin{align*}
    (X^x-X^y)^2_t \leq (x-y)^2 \prod_{u=s+1}^t \left(1-(2 \, \alpha_u \, m(W_u) - \alpha_u^2 \, M(W_u)^2)\right)
  \end{align*}
  and the result follows by taking expectation, using the fact that
  the $W_t$ are i.i.d. Note that the square root is well-defined since
  $1-(2 \, \alpha_u \, \bar{m} - \alpha_u^2 \, \bar{M}^2) \geq 1 -
2 \, \alpha_u \, \bar{M} + \alpha_u^2\bar{M}^2 = (1-\alpha_u\bar{M})^2 \geq 0$.
\end{proof}

\begin{corollary}\label{corr:polyakkappa}
  If $\alpha_1 + \alpha_T \leq \frac{2\bar{m}}{\bar{M}^2}$ and
  $\alpha_{t+1} \leq \alpha_t$ for all $t \geq 1$ then for any
  1-Lipschitz function $f$
  \begin{equation*}
    \Abs{P_{s,t} f(x) - P_{s,t} f(y)} \leq \Abs{x -
      y} \, \left(1-\alpha_T \, \bar{m} + \tfrac12 \alpha_T^2 \bar{M}^2\right)^{t-s}
  \end{equation*}
\end{corollary}
\begin{proof}
  Let $\beta_t = \sqrt{1-2\alpha_t\bar{m} + \alpha_t^2
    \bar{M}^2}$.
  Since $\beta_u^2 - \beta_T^2 = (\alpha_u -
  \alpha_T)((\alpha_u+\alpha_T)\bar{M}^2 - 2 \bar{m})$,
  the assumptions on $\alpha$ imply that $\beta_u \leq \beta_T$ for
  all $1 \leq u \leq T$ so that
  \begin{equation*}
    \prod_{u=s+1}^{t} \beta_u \leq \beta_T^{t-s}.
  \end{equation*}
  Since $\sqrt{1-x} \leq 1 - x/2$ we also have
  \begin{equation*}
    \beta_T \leq 1 - \alpha_T \, \bar{m} + \tfrac12 \alpha_T^2 \bar{M}^2.
  \end{equation*}
  Together with the preceding Lemma we finally obtain
  \begin{align*}
    \Abs{P_{s,t} f(x) - P_{s,t} f(y)}
    &\leq \llipnorm{f}\, \Abs{x - y} \, \prod_{u=s+1}^t \beta_u
    &\leq \Abs{x -
      y} \, \left(1-\alpha_T \, \bar{m} + \tfrac12 \alpha_T^2 \bar{M}^2\right)^{t-s}.
  \end{align*}
\end{proof}

\begin{corollary} \label{corr:PolyakRuppert}
  For any $T \in \mathbb{N}$ fixed, if $\alpha_t = \lambda t^{-p}$ for
  $p < 1/2$ and $\lambda$ such that
  \begin{equation*}
    \lambda \leq \frac{2 \bar{m}}{\bar{M}^2} \, \frac{T^p}{1+T^p}
  \end{equation*}
  we have
  \begin{multline*}
  \PP\left(\frac{1}{T}\sum_{t=0}^{T-1} X_t - \E\left[\frac{1}{T}\sum_{t=0}^{T-1} X_t\right] \geq R \:;\: \sup_{t\leq T} \sup_w \Abs{g(X_t, w)}\leq
    G\right) \\
  \leq \exp\left(-\frac{(1-2p) \, \bar{m}^2 \, R^2 \, T}{32 \, G^2}\right).
\end{multline*}
\end{corollary}
\begin{proof}
  We are in the situation of Proposition~\ref{prop:discrete} with
  $\sigma_t^2 = 1$, $C_t = \lambda t^{-p} g^*(X_t)$ from
  Lemma~\ref{lemma:polyakCt} and
  $\kappa_t = \kappa_T = \lambda T^{-p} (\bar{m} - \tfrac12 \lambda T^{-p} \bar{M}^2)$
  from Corollary~\ref{corr:polyakkappa}.  On the set
  $\{ \sup_{t\leq T} \Abs{g^*(X_t)}\leq G \}$ we have
  $C_t^2 \leq \lambda^2 G^2 t^{-2p}$ and
  \begin{equation*}
    \sum_{t=1}^T \frac{\sigma_t^2 C_t^2}{\kappa_t^2} \leq
    \frac{\lambda^2 G^2}{\kappa_T^2} \sum_{t=1}^T
    t^{-2p} \leq \frac{\lambda^2 G^2}{\kappa_T^2}
    \int_0^T t^{-2p} dt = \frac{\lambda^2 G^2
      T^{1-2p}}{(1-2p)\kappa_T^2} \leq \frac{4 
      G^2 T}{(1-2p)\bar{m}^2}
  \end{equation*}
  where the last inequality follows from the assumption that $\lambda
  \leq \frac{2 \bar{m}}{\bar{M}^2} \, \frac{T^p}{1+T^p} \leq
  \frac{\bar{m}}{\bar{M}^2} T^p$ so that 
  \begin{equation*}
    \kappa_T
    = \lambda T^{-p} (\bar{m} - \tfrac12 \lambda T^{-p} \bar{M}^2)
    \geq \tfrac12 \lambda T^{-p} \bar{m}.
  \end{equation*}
  It remains to apply Proposition~\ref{prop:discrete}:
  \begin{align*}
   \lefteqn{ \PP\left(\sum_{t=0}^{T-1} X_t - \E \sum_{t=0}^{T-1} X_t \geq RT \:;\: \sup_{t\leq T} \Abs{g^*(X_t)}\leq
    G\right)}\quad\\
    &\leq \PP\left(\sum_{t=0}^{T-1} X_t - \E\sum_{t=0}^{T-1} X_t \geq RT\:;\: \sum_{t=1}^T
      \frac{\sigma_t^2 C_t^2}{\kappa_t^2} \leq \frac{4 
      G^2 T}{(1-2p)\bar{m}^2}\right) \\
    &\leq \exp\left(-\frac{(1-2p) \, \bar{m}^2 \, R^2 \, T}{32 \,G^2}\right).
  \end{align*}
\end{proof}

\subsection{Lipschitz observables and SDEs contractive at infinity}\label{example:sde}
We use the notations $\llipnorm{f} = \sup_{x\neq y}
\frac{\abs{f(x)-f(y)}}{\abs{x-y}}$ for the Lipschitz seminorm,
\begin{equation*}
W_1(\nu_1, \nu_2) = \sup_{f: \llipnorm{f}\leq1} \left(\int f d\nu_1 - \int f d\nu_2\right)
\end{equation*}
for the $L^1$ transportation distance, $\mu P_{s,t} = \int P_{s,t}(x,
\cdot) \mu(dx)$ and $\mu(f) = \int f d\mu$ for a function $f$, measures
$\mu, \nu_1, \nu_2$ and a transition kernel $P_{s,t}$.

Consider the SDE
\begin{align*}
  dX_t = b(t, X_t)\,dt + dB_t, \quad X_0 \sim \nu
\end{align*}
with $b\colon [0, \infty) \times \R^d \to \R^d$ a locally Lipschitz
continuous function, $B$ a $d$-dimensional Brownian motion and $\nu$ a
probability measure on $\R^d$.

We make the following additional assumption on
``contractivity at infinity'': there exist constants $D, K > 0$ such
that for all $t \geq 0$ and $x, y \in \R^d$ with $\abs{x-y} > D$ we have
\begin{equation*}
  (x-y)\cdot(b(t, x) -b(t, y)) \leq -K \abs{x-y}^2.
\end{equation*}

In~\autocite{eberle_reflection_2015} it was shown (in the
time-homogeneous setting, but the methods extend directly to the
time-inhomogeneous case \autocite{cheng_exponential_2020,}) that the assumption
on $b$ implies the exponential contractivity of the transition kernels
$P_{s,t}$ associated to $X$ in $L^1$
transportation distance: there exist constants $\rho, \kappa > 0$ such that for any two
probability measures $\nu_1$ and $\nu_2$ on $\R^d$ we have
\begin{equation}\label{eq:exsdeW1}
  W_1(\nu_1 P_{s,t}, \nu_2 P_{s,t}) \leq \rho e^{-\kappa (t-s)} \, W_1(\nu_1, \nu_2)
\end{equation}
or equivalently (\autocite{kuwada_duality_2010,}) for all Lipschitz
functions $f$
\begin{equation*}
  \norm{\nabla P_{s,t} f}_\infty \leq \rho e^{-\kappa (t-s)} \llipnorm{f}.
\end{equation*}

The key estimate~\eqref{eq:exsdeW1}, and the results of this section, hold
in fact for a large class of SDEs ``contractive at infinity''. The
work~\autocite{eberle_reflection_2015} already includes the case of a
diffusion coefficient which is not the identity as well as explicit
values for the constants, see
also~\autocite{donati-martin_semi_2014}. The
paper~\autocite{wang_exponential_2020} treats the case with a
non-constant diffusion matrix and generalizes the results to
Riemannian manifolds. For the non-autonomous situation, in the general
context of Riemannian manifolds with possibly time-dependent metric
and with explicit constants,
see~\autocite{cheng_exponential_2020}. Another approach, which
provides exponential gradient estimates for SDEs with highly
degenerate diffusion matrices, can be found
in~\autocite{crisan_uniform_2019,}.


We assume furthermore that $\nu$ satisfies a $T_1$ inequality \autocite{wu_transportation_2004}: there
exists a constant $C$ such that for any 1-Lipschitz function $f$ and
$\lambda > 0$, we have
\begin{equation}\label{eq:exsdenu}
  \int e^{\lambda \left(f - \int f d\nu\right)} d\nu \leq e^{\frac{\lambda^2 C}{2}}.
\end{equation}
Note that for $\nu = \delta_x$ the $T_1$ inequality holds with
constant $C = 0$.

We are going to show that whenever~\eqref{eq:exsdeW1}
and~\eqref{eq:exsdenu} hold then for all Lipschitz functions $f$
and $R >0$

\begin{multline*}
  \PP\left(\frac{1}{T}\int_0^T f(X_t)\,dt -
    \E\left[\frac{1}{T}\int_0^T f(X_t)\,dt\right] \geq R\right) \\ \leq
  \exp\left({-\frac{\kappa^2 R^2 T}{2 \rho^2 \llipnorm{f}^2 \left(1 +
      C\,\frac{1-e^{-\kappa T}}{T}\right)}}\right).
\end{multline*}

In the time-homogeneous setting, $X$ has a unique stationary measure
$\mu$ and we have
\begin{multline*}
  \PP\left(\frac{1}{T}\int_0^T f(X_t)\,dt - \int f \,d\mu \geq R +
       \rho \llipnorm{f} \frac{1-e^{-\kappa T}}{\kappa T} W_1(\mu, \nu)\right) \\ \leq
  \exp\left({-\frac{\kappa^2 R^2 T}{2 \rho^2 \llipnorm{f}^2 \left(1 +
      C\,\frac{1-e^{-\kappa T}}{T}\right)}}\right).
\end{multline*}

In the time-inhomogeneous setting, we have the existence of a unique
evolution system of measures $\mu_t$ (\autocite{da_prato_note_2007})
such that
\begin{multline*}
  \PP\left(\frac{1}{T}\int_0^T f(X_t)\,dt - \frac1T \int_0^T \mu_t(f) \,dt \geq R +
       \rho \llipnorm{f} \frac{1-e^{-\kappa T}}{\kappa T} W_1(\mu_0, \nu)\right) \\ \leq
  \exp\left({-\frac{\kappa^2 R^2 T}{2 \rho^2 \llipnorm{f}^2 \left(1 +
      C\,\frac{1-e^{-\kappa T}}{T}\right)}}\right).
\end{multline*}

\begin{lemma}\label{lemma:QVGamma}
  For any two continuously differentiable functions $f, g$ on $\R^n$
  \begin{equation*}
    \pQV{f(X), g(X)}_t = \frac12 \int_0^t \nabla f(X_s) \cdot \nabla g (X_s)\,ds.
  \end{equation*}
\end{lemma}
\begin{proof}
  By polarization it is sufficient to show that
  \begin{equation*}
    \pQV{f(X)}_t = \frac12 \int_0^t \abs{\nabla f(X_s)}^2\,ds.
  \end{equation*}
  For $f$ twice continuously differentiable, the result follows
  directly from Itô's formula. To extend the result to general $f$ by
  approximation, define the sequence of stopping times
  $\tau_k = \inf\{t \geq 0: \abs{X_t} \geq k\}$. For each $k$ fixed, choose
  a sequence of smooth, compactly supported functions $f_n$ such that
  $f_n \to f$ and $\abs{\nabla f_n} \to \abs{\nabla f}$ uniformly on
  $\{x: \abs{x} \leq k\}$. Then $f_n(X_{t\wedge\tau_k})$ converges to
  $f(X_{t\wedge\tau_k})$ uniformly on compacts in probability (u.c.p.), meaning
  that $\sup_{t \in [0,T]} \abs{f(X_{t\wedge\tau_k}) - f_n(X_{t\wedge\tau_k})} \to
  0$ in probability for any $T \geq 0$. Indeed, $\sup_t \abs{f(X_{t\wedge\tau_k}) - f_n(X_{t\wedge\tau_k})} \leq
  \sup_{\abs{x}\leq k} \abs{f(x) - f_n(x)} \to 0$ as $n\to\infty$. By
  continuity of the predictable quadratic variation
  \begin{equation*}
    \pQV{f_n(X)}_{t\wedge\tau_k} \to
      \pQV{f(X)}_{t\wedge\tau_k} \text{ u.c.p.} 
  \end{equation*}
  On the other hand, by the uniform convergence of $\abs{\nabla f_n}$ to
  $\abs{\nabla f}$ on $\{x: \abs{x} \leq k \}$,
  \begin{equation*}
    \pQV{f_n(X)}_{t\wedge\tau_k} = \frac12 \int_0^{t\wedge\tau_k}
    \abs{\nabla f_n(X_s)}^2\,ds \to \frac12 \int_0^{t\wedge\tau_k}
    \abs{\nabla f(X_s)}^2\,ds \text{ a.s.}
  \end{equation*}
  so that
  \begin{equation*}
    \pQV{f(X)}_{t\wedge\tau_k} =  \frac12 \int_0^{t\wedge\tau_k}
    \abs{\nabla f(X_s)}^2 \,ds.
  \end{equation*}
  Now the result follows by letting $k \to \infty$.
\end{proof}

\begin{lemma}
  For a 1-Lipschitz function $f$, let
  $M^T_t = \E^{\F_t} \int_0^T f(X_s)\,ds$. Then
  \begin{equation*}
    \pQV{M^T}_T \leq \frac{\rho^2 T}{\kappa^2}.
  \end{equation*}
\end{lemma}
\begin{proof}
  From~\eqref{eq:exsdeW1} it follows that $P_t f$ is continuously
  differentiable and the Lipschitz condition on $f$ ensures that
  $\E f(X_t)^2 < \infty$ for all $t$. By the preceding Lemma we can
  apply Proposition~\ref{prop:Markov} with
  $\Gamma(f, g) = \frac12 \nabla f \cdot \nabla g$. Since by
  Cauchy-Schwarz $2 \Gamma(f, g) \leq \abs{\nabla f} \abs{\nabla g}$
  we get
  \begin{align*}
    d\pQV{M^T}_t &= \int_t^T \int_t^T 2 \Gamma(P_{t,u} f, P_{t,v}
                   f)(X_t)\, du \, dv \\
    &\leq \left(\int_t^{T} \Abs{\nabla P_{t,u}
      f(X_t)}\,du\right)^2\,dt 
     \leq \rho^2 \left(\int_0^{T-t} e^{-\kappa u}\,du\right)^2 \,dt
    \leq \frac{\rho^2}{\kappa^2} \, dt
  \end{align*}
  and $\pQV{M^T}_T \leq \frac{\rho^2 T}{\kappa^2}$.
\end{proof}

\begin{lemma}
  For all 1-Lipschitz functions $f$ we have
  \begin{equation*}
    \Lambda(\lambda) \leq \frac{\lambda^2 C\rho^2}{2} \left(\frac{1-e^{-\kappa T}}{\kappa}\right)^2
  \end{equation*}
  where
  \begin{equation*}
    \Lambda(\lambda) = \log
    \E\, e^{\lambda\left(\E^{\F_0} S_T - \E S_T\right)}.
  \end{equation*}
  and
  \begin{equation*}
    S_T = \int_0^T f(X_t) dt.
  \end{equation*}
\end{lemma}
\begin{proof}
  We have by Theorem~\ref{thm:main} and the Markov property
  \begin{equation*}
    \E^{\F_0} S_T - \E S_T = \int_0^T P_{0,t}f(X_0)\,dt - \E \int_0^T
    P_{0,t}f(X_0) \,dt
    = F(X_0) - \int F d\nu
  \end{equation*}
  with
  \begin{equation*}
    F(x) = \int_0^T P_{0,t}f(x) \,dt.
  \end{equation*}
  From the triangle inequality for the Lipschitz seminorm and~\eqref{eq:exsdeW1} we get
  \begin{equation*}
    \llipnorm{F} \leq \int_0^T \llipnorm{P_{0,t} f} \, dt \leq \int_0^T
    \rho e^{-\kappa t} \,dt = \rho \frac{1-e^{-\kappa T}}{\kappa}
  \end{equation*}
  and the result follows directly from~\eqref{eq:exsdenu}.
\end{proof}

\begin{proposition}
  For all Lipschitz functions $f$ and $R >0$ we have
\begin{multline*}
  \PP\left(\frac{1}{T}\int_0^T f(X_t)\,dt -
    \E\left[\frac{1}{T}\int_0^T f(X_t)\,dt\right] \geq R\right) \\ \leq
  \exp\left({-\frac{\kappa^2 R^2 T}{2 \rho^2 \llipnorm{f}^2 \left(1 +
      C\,\frac{1-e^{-\kappa T}}{T}\right)}}\right).
\end{multline*}
\end{proposition}
\begin{proof}
  The result follows directly from the preceding Lemmas and
  Corollary~\ref{corr:subGaussian}. Indeed, noting that
  $f / \llipnorm{f}$ is 1-Lipschitz and that $H^0_T = \pQV{M^T}_T$,
  \begin{align*}
    \lefteqn{\PP\left(\frac{1}{T}\int_0^T f(X_t)\,dt -
    \E\left[\frac{1}{T}\int_0^T f(X_t)\,dt\right] \geq R\right)}\quad \\
    & = \PP\left(\int_0^T f(X_t) / \llipnorm{f}\,dt -
      \E \int_0^T f(X_t) / \llipnorm{f}\,dt \geq RT / \llipnorm{f} \; ;\; H^0_T
      \leq \frac{\rho^2 T}{\kappa^2}\right)  \\
    & \leq
  \exp\left({-\frac{\kappa^2 R^2 T}{2 \rho^2 \llipnorm{f}^2 \left(1 +
      C\,\frac{1-e^{-\kappa T}}{T}\right)}}\right).
  \end{align*}
\end{proof}

\begin{lemma}
  The process $X$ admits an evolution system of measures $\mu_t$ such that
  for any 1-Lipschitz function $f$ we have
  \begin{equation*}
    \E\left[\frac{1}{T} \int_0^T f(X_t)\,dt\right] - \frac1T\int_0^T \mu_t(f)\,dt
    \leq \rho \frac{1-e^{-\kappa T}}{\kappa T} W_1(\nu, \mu_0).
  \end{equation*}
  In the time-homogeneous case $\mu_t = \mu$ is the stationary measure
  associated to $X$.
\end{lemma}
\begin{proof}
  Existence and uniqueness of an evolutionary system of measures is a
  special case of Theorem~3.5 in~\autocite{cheng_exponential_2020}. In
  the time-homogeneous case existence and uniqueness of a stationary
  measure is also shown in~\autocite{eberle_reflection_2015} and is a
  direct consequence of~\eqref{eq:exsdeW1}.
  
  Since from the definition of an evolutionary system of measures
  $\mu_t = \mu_0 P_{0,t}$ we have by~\eqref{eq:exsdeW1}
  \begin{align*}
   \lefteqn{ \E \left[ \frac{1}{T} \int_0^T f(X_t)\,dt -
    \frac1T\int_0^T \mu_t(f)\,dt \right]}\quad \\
    &= \frac{1}{T} \int_0^T \left(\nu P_{0,t} f  - \mu_0 P_{0,t} f\right)\,dt
    \leq \frac{1}{T} \int_0^T W_1(\nu P_{0,t}, \mu_0 P_{0,t})\,dt \\
    &\leq \frac{1}{T} \int_0^T \rho e^{-\kappa t} W_1(\nu, \mu_0)\,dt
    \leq \rho \frac{1-e^{-\kappa T}}{\kappa T} W_1(\nu, \mu_0).
  \end{align*}
\end{proof}

\begin{proposition}\label{prop:sde:mu}
  For all Lipschitz functions $f$ and $R > 0$ we have
  \begin{multline*}
  \PP\left(\frac{1}{T}\int_0^T f(X_t)\,dt - \frac1T \int_0^T \mu_t(f) \,dt \geq R +
       \rho \llipnorm{f} \frac{1-e^{-\kappa T}}{\kappa T} W_1(\mu_0, \nu)\right) \\ \leq
  \exp\left({-\frac{\kappa^2 R^2 T}{2 \rho^2 \llipnorm{f}^2 \left(1 +
      C\,\frac{1-e^{-\kappa T}}{T}\right)}}\right).
\end{multline*}
\end{proposition}
\begin{proof}
  By the preceding Lemma
  \begin{multline*}
    \PP\left(\frac{1}{T}\int_0^T f(X_t)\,dt -  \frac1T \int_0^T \mu_t(f) \,dt \geq R +
    \rho \llipnorm{f} \frac{1-e^{-CT}}{CT} W_1(\mu_0, \nu)\right) \\
     \leq \PP\left(\frac1T\int_0^T f(X_t)\,dt -
      \E \left[\frac1T\int_0^T f(X_t)\,dt\right] \geq R\right)
  \end{multline*}
  and the result follows immediately from the preceding Proposition.
\end{proof}

\subsection{Martingale Integrands}\label{sec:examples:mart}
\begin{proposition}
  For a Brownian motion $B$ we have for all $R, T \geq 0$
  \begin{equation*}
    \PP\left(\frac1{T^2}\int_0^T B_t \, dt \geq R\right)
    \leq \exp(-3 R^2 T).
  \end{equation*}
\end{proposition}
\begin{proof}
  Let $M_t^T = \E^{\F_t} \int_0^T B_u\,du$. Then by Proposition~\ref{prop:Martingale}
  we have
  \begin{equation*}
    \pQV{M^T}_T = \int_0^T (T-t)^2 \, d\pQV{B}_t = \int_0^T (T-t)^2 \,
    dt = \frac{T^3}3 
  \end{equation*}
  and by Corollary~\ref{corr:subGaussian} with $H^0_T = \pQV{M^T}_T$
  \begin{equation*}
    \PP\left(\frac1{T^2}\int_0^T B_t \, dt \geq R\right)
    = \PP\left(\int_0^T B_t \, dt \geq R T^2\;; H^0_T \leq
      \frac{T^3}{3}\right) \\
    \leq \exp(-3 R^2 T).
  \end{equation*}
\end{proof}

\begin{proposition}
  For a Poisson process $N$ we have for all $R,T > 0$
  \begin{equation*}
    \PP\left(\frac1{T^2}\int_0^T N_t \, dt - \frac12 \geq R\right)
    \leq \exp\left(-\frac T 3 h\bigl(3R\bigr)\right)
    \leq \exp\left(-\frac {3 R^2 T}{2(1+R)}\right)
  \end{equation*}
  with
  \begin{equation*}
    h(x) = (1+x)\log(1+x) - x.
  \end{equation*}
\end{proposition}
\begin{proof}
  Let $X_t = N_t - t$ and $M_t^T = \E^{\F_t} \int_0^T X_t\,dt$.
  We have by Proposition~\ref{prop:Martingale}
  \begin{equation*}
    \Delta M^T_t = (T-t)\Delta X_t \leq T, \quad 0 \leq t \leq T
  \end{equation*}
  and
  \begin{equation*}
    \pQV{M^T}_T = \int_0^T (T-t)^2 \, d\pQV{X}_t = \int_0^T (T-t)^2 \,
    dt = \frac{T^3}3 
  \end{equation*}
  so that
  \begin{equation*}
    H^T_T = \pQV{M^T}_T + \sum_0^T (\Delta M^T_t)^2 \one\{\Delta M^T_t
    > T\} = \pQV{M^T}_T = \frac{T^3}3
  \end{equation*}
  and by Corollary~\ref{corr:subPoisson}
  \begin{equation*}
    \PP\left(\frac1{T^2}\int_0^T N_t \, dt - \frac12 \geq R\right)
    = \PP\left(\int_0^T X_t \, dt \geq R T^2\;; H^T_T \leq
      \frac{T^3}{3}\right) \\
    \leq \exp\left(-\frac T 3 h\bigl(3R\bigr)\right).
  \end{equation*}
  The second inequality in the result follows immediately from the
  elementary inequality $h(x) \geq \frac{x^2}{2(1+x/3)}$ for $x > 0$.
\end{proof}

\begin{proposition}
  For a Brownian motion $B$ we have for all $R,T \geq 0$
  \begin{equation*}
    \PP\left(\Abs{\frac{1}{T^2}\int_0^T B_t^2\,dt - \frac12} \geq R\right)
    \leq 2^{1/3}\left(1 \wedge \frac{4R + 3/4}{3R^2}\right)
    \exp\left(-\frac{R^2}{3(4R + 3/4)}\right).
  \end{equation*}
\end{proposition}
\begin{proof}
  Consider the local martingale $X_t = B_t^2 - t$. We have using
  integration by parts that
  \begin{equation*}
    X_t = 2 \int_0^t B_s \, dB_s
  \end{equation*}
  so that
  \begin{equation*}
    \pQV{X}_t = 4 \int_0^t B_s^2 \,ds = 4 \int_0^t 
    (X_s + s)\,ds.
  \end{equation*}
  By~\eqref{eq:QVcont} and Remark~\ref{remark:martT2} we have
  \begin{equation*}
    \pQV{M^T}_T = \int_0^T (T-t)^2 \, d\pQV{X}_t = 4 \int_0^T (T-t)^2\,
    (X_t + t)\,dt
  \end{equation*}
  and
  \begin{equation*}
    \E \pQV{M^T}_T = 4 \int_0^T (T-t)^2 \,t \,dt.
  \end{equation*}
  Finally
  \begin{align*}
    \pQV{M^T}_T + \E\pQV{M^T}_T &= 4 \int_0^T (T-t)^2\, X_t\,dt + 8
    \int_0^T (T-t)^2\, t \,dt \\ & \leq 4 T^2 \int_0^T X_t\, dt + (3/4)T^4
  \end{align*}
  and the result follows from Corollary~\ref{corr:selfnormCD} with $C = 4T^2$
  and $D = (3/4)T^4$.
\end{proof}

\subsection{Squared Ornstein-Uhlenbeck Process}\label{sec:squaredOU}
Let $X^x$ be the Ornstein-Uhlenbeck process on $\R^d$, solution to the SDE
\begin{equation*}
  dX^x_t = -\kappa X^x_t + dB_t, \quad X^x_0 = x
\end{equation*}
with $\kappa > 0$ and $B_t$ a d-dimensional Brownian motion.

In this section, we will derive concentration inequalities for
additive functionals with the square of $X^x$ as integrand. This case
is challenging since for $\phi(x) = \abs{x}^2$, $\nabla P_t \phi(x)$
cannot be bounded uniformly in $x$. We will make use of the special
properties of the Ornstein-Uhlenbeck semigroup. In the next section we
will extend the approach to a slightly more general situation, which
however does not recover the bound developed in this section. The case
of the squared Ornstein-Uhlenbeck process was previously studied in
\autocite{lezaud_chernoff_2001} (Example 4.2) and
\autocite{gao_bernstein-type_2014} (Example 3.1) using analytic
methods, which require the initial law of $X$ to be absolutely
continuous with respect to the stationary measure of $X$.

\begin{proposition}\label{prop:OU2}
  We have for all $R, T > 0, x \in \R^d$
\begin{multline*}  
  \PP\left(\Abs{\frac1T \int_0^T \abs{X^x_t}^2\,dt - \E\left[
  \frac1T \int_0^T
  \abs{X^x_t}^2\,dt\right]} \geq R\right) \\
  \leq 2^{1/3} \left(1 \wedge \frac{R + D}{3\kappa^2 R^2
    T}\right)^{1/3} \exp\left(-\frac{\kappa^2 R^2 T}{3(R+D)}\right)
\end{multline*}
with
\begin{equation*}
  D = \frac{\abs{x}^2}{\kappa T} + \frac{d}{\kappa}
\end{equation*}
\end{proposition}

\begin{proof}
  We have component-wise
  $(X^x_t)_i = x_i e^{-\kappa t} + \int_0^t e^{-\kappa(t-s)}\,dB^i_s$.
  Let $\phi(x) = \abs{x}^2$ and for $t \geq 0$ set
  $P_t \phi(x) = \E \phi(X_t) = \E \abs{X_t}^2$. Then for
  $\abs{x} < 1/\varepsilon$ for some arbitrary $\varepsilon > 0$, we can
  differentiate under the expectation
  \begin{align*}
    \partial_i \E \phi(X^x_t) = \E \partial_i \phi(X^x_t) = 2 x_i
    e^{-2\kappa t}, \quad \abs{x} < 1/\varepsilon
  \end{align*}
  so that
  \begin{equation*}\label{eq:OU2nabla}
    \nabla P_t \phi(x) = 2 x e^{-2\kappa t}, \quad \abs{x} \leq 1/\varepsilon.
  \end{equation*}
  By Lemma~\ref{lemma:QVGamma}, for any two continuously differentiable
$f, g$ we have
\begin{equation*}
  \pQV{f(X), g(X)}_t = \frac12 \int_0^t \nabla f(X_s) \nabla g(X_s)\,ds
\end{equation*}
so that by Proposition~\ref{prop:Markov}, on $A_\varepsilon := \{\sup_{0\leq t\leq T}
\abs{X^x_t} < 1/\varepsilon\}$,
\begin{align*}
  d\pQV{M^T}_t
  &= \int_t^T \int_t^T (\nabla P_{u-t} \phi \cdot \nabla P_{v-t}
    \phi)(X_t) \,du \,dv \,dt \\
  &= 4 \abs{X_t}^2 \left(\int_0^{T-t} e^{-2 \kappa u} du\right)^2 \,dt\\
  &= \frac{X_t^2}{\kappa^2} \left(1 - e^{-2 \kappa (T-t)}\right)^2 \,dt.
\end{align*}
Integrating, we get
\begin{equation*}
  \pQV{M^T}_T = \frac{1}{\kappa^2} \int_0^T \abs{X_t}^2 \left(1 - e^{-2 \kappa
  (T-t)}\right)^2 dt \leq \frac{1}{\kappa^2} \int_0^T \abs{X_t}^2 \,dt
\end{equation*}
so that Corollary~\ref{corr:selfnormCSTD} applies with $C =
\frac{1}{\kappa^2}$ and $D = 0$:
\begin{align*}
  \lefteqn{\PP\left(\Abs{\int_0^T \abs{X_t}^2\,dt - \E \int_0^T
  \abs{X_t}^2\,dt} \geq RT\;; A_\varepsilon\right)} \quad \\
  &= \PP\left(\Abs{\int_0^T \abs{X_t}^2\,dt - \E \int_0^T
  \abs{X_t}^2\,dt} \geq RT\;; H^0_T \leq \frac{1}{\kappa^2} \int_0^T
    \abs{X_t^2} \,dt\right) \\
  & \leq 2^{1/3} \left(1 \wedge \frac{R + D'}{3\kappa^2 R^2 T}\right)^{1/3}
                                                                                            \exp\left(-\frac{\kappa^2
    R^2
    T}{3(R+D')}\right)\\
  & \text{ with } D' = \abs{\E S_T}/T + (\kappa^2/T) \E H^0_T \leq 2 \abs{\E S_T}/T.
\end{align*}
Since
\begin{equation*}
  \E \abs{X^x_t}^2 = \abs{x}^2 e^{-2\kappa t} + d \frac{1-e^{-2\kappa t}}{2\kappa}
\end{equation*}
we have that
\begin{equation*}
  D' \leq \frac2T \int_0^T \E\abs{X^x_t}^2\,dt \leq
  \frac{\abs{x}^2}{\kappa T} + \frac{d}{\kappa}.
\end{equation*}
We can lift the restriction to $A_\varepsilon$ by noting that by
Markov's inequality and Doob's maximal inequality
\begin{align*}
  \PP(A_\varepsilon^c) &= \PP(\sup_t \abs{X_t} > 1/\varepsilon) \leq \PP(\sup_t e^{\kappa
    t}\abs{X_t} > 1/\varepsilon) = \PP\left(\sup_t \Abs{\int_0^t e^{\kappa
      s} dB_s} > 1/\varepsilon\right) \\
  & \leq \varepsilon^2 \E \sup_t \Abs{\int_0^t e^{\kappa
      s} dB_s}
  \leq 4 \varepsilon^2 \, \frac{e^{2\kappa T} - 1}{\kappa}
\end{align*}
so that finally
\begin{align*}
\lefteqn{\PP\left(\Abs{\int_0^T \abs{X_t}^2\,dt - \E \int_0^T
  \abs{X_t}^2\,dt} \geq RT\right)} \quad \\
  &\leq \PP\left(\Abs{\int_0^T \abs{X_t}^2\,dt - \E \int_0^T
  \abs{X_t}^2\,dt} \geq RT\;; A_\varepsilon\right) + \PP(A_\varepsilon^c)
  \\
  &\leq 2^{1/3} \left(1 \wedge \frac{R + D}{3\kappa^2 R^2
    T}\right)^{1/3} \exp\left(-\frac{\kappa^2 R^2 T}{3(R+D)}\right) + O(\varepsilon^2)
\end{align*}
with
\begin{equation*}
  D = \frac{\abs{x}^2}{\kappa T} + \frac{d}{\kappa}
\end{equation*}
and the result follows by letting $\varepsilon \to 0$.
\end{proof}

\begin{remark}
  From the previous proposition, we also get that
  \begin{equation*}\label{eq:LDPOU2}
    \lim_{T\to\infty} T^{-1} \log \PP\left(\frac{1}{T} \int_0^T \abs{X^x_t}^2
      \geq R\right) \leq -I(R)
  \end{equation*}
  with
  \begin{equation*}
    I(R) = \frac{(d - 2\kappa R)^2}{12 (R + d/(2\kappa))}.
  \end{equation*}
  According to large deviation estimates
  from~\autocite{bryc_large_1997,}, the optimal bound in the
  right-hand side of~\eqref{eq:LDPOU2} is obtained by replacing $I$
  with the good rate function
  \begin{equation*}
    J(R) = \frac{(d - 2\kappa R)^2}{8 R}.
  \end{equation*}
  Tracing back the computations, we see that the discrepancy between
  the denominators of $I$ and $J$, namely the factor $12$ instead of
  $8$ and the extra term $d/(2\kappa)$ for $I$, can be directly
  attributed to the factor $2/3$ and term $\E[A^2]$ in the
  self-normalized martingale inequality~\eqref{eq:selfnormAB}. This
  discrepancy is the same as the one observed between the sharp bound
  for normal random variables and the self-normalized bound
  in~\autocite{de_la_pena_exponential_2009,}, see also Remark 2.2 in
  that paper. This suggests that the the martingale $M^T$ belongs to a
  subclass of martingales for which the bound
  in~\autocite{de_la_pena_exponential_2009,} can be sharpened.
\end{remark}

\subsection{Squared Lipschitz Integrands}\label{sec:squaredLip}
Let $X$ be a Markov process with generator $L$,
transition kernel $P_{s,t}$ and squared field operator $\Gamma$, i.e.
$\Gamma(f) = Lf^2 - 2 fLf$.

We assume the following commutation property between the squared field
operator and the transition kernel:
\begin{equation*}
  \Gamma(P_{s,t} f) \leq \sigma^2 \left(P_{s,t}\sqrt{\Gamma(f)}\right)^2\, e^{-2\kappa(t-s)}
\end{equation*}
for some $\sigma, \kappa \geq 0$.

Using the inequalities on self-normalized martingales, we can derive a
Bernstein-type inequality for time averages of positive functions $g^2$
such that $\Gamma(g^2) \leq 2 g^2$.

\begin{proposition}
For all twice continuously differentiable $g^2$ such that
$\Gamma(g^2) \leq 2 g^2$ and $Lg^2 \leq -2C g^2 + D^2$ for
some constants $C \in \R, D \geq 0$ we have for
\begin{equation*}
  S_T = \int_0^T g^2(X_t)\,dt
\end{equation*}
the following Bernstein-type inequality:
\begin{multline*}
  \PP\left(\abs{S_T - \E S_T} \geq RT\right) \\
  \leq 2^{1/3} \exp\left(-\frac{R^2 T}{24 \sigma^2 \left(C_\kappa^2(T)R + 2
        C_\kappa^2(T) \E (S_T/T)  + 2
        D_{\kappa,C}^2(T)\right)}\right)
\end{multline*}
with
\begin{align*}
  C_\kappa(T) &= \int_0^T e^{-(\kappa+C)u}\,du \\
  D_{\kappa,C}(T) &= D \int_0^T e^{-\kappa u} \int_0^u e^{-C(u-v)}\,dv \,du.
\end{align*}
\end{proposition}
\begin{proof}
Since $g^2$ is assumed twice continuously differentiable, we have
$\partial_u P_{t,u}g^2 = P_{t,u} Lg^2$ and we get from our assumption
$Lg^2 \leq -2C g^2 + D^2$ by Gronwall's lemma that
\begin{equation*}
  P_{t,u}g^2 \leq g^2 e^{-2C(u-t)} + D^2 \int_t^u e^{-2C (u-v)}\,dv.
\end{equation*}
Together with the assumption that
$\Gamma(g^2) \leq 2 g^2$ we have
\begin{equation*}
  P_{t,u} (2 \Gamma(g^2)) \leq 4 P_{t,u} (g^2) \leq 4 g^2
  e^{-2C(u-t)} + D^2 \int_t^u e^{-2C (u-v)}\,dv
\end{equation*}
and by applying the assumption on commutation of $\Gamma$ and $P$
\begin{align*}
  2\Gamma(P_{t,u} g^2) & \leq \sigma^2 (P_{t,u}\sqrt{2\Gamma(g^2)})^2
  e^{-2\kappa(u-t)} \leq \sigma^2 P_{t,u} (2 \Gamma(g^2)) e^{-2\kappa(u-t)}
  \\ & \leq 4 \sigma^2 g^2\,
  e^{-2(\kappa+C)(u-t)} + 4 \sigma^2 D^2 \,e^{-2\kappa(u-t)} \int_t^u e^{-2C (u-v)}\,dv.
\end{align*}
Now from Proposition~\ref{prop:Markov} we have for all $T > 0$
\begin{align*}
  d\pQV{M^T}_t
  &= \int_t^T \int_t^T 2 \Gamma(P_{t,u} g^2, P_{t,v}
                 g^2)(X_t)\,dv\,du\,dt \\
  &\leq \left(\int_t^T \sqrt{2 \Gamma(P_{t,u} g^2)(X_t)} \,du\right)^2
    \,dt \\
  &\leq 4 \sigma^2 \left(\sqrt{g^2(X_t)} \int_t^T e^{-(\kappa+C)(u-t)}
    \,du + D \int_t^T \,e^{-\kappa(u-t)} \int_t^u e^{-C (u-v)}\,dv
    \,du \right)^2 \,dt \\
  &\leq 8 \sigma^2 \left(C_{\kappa}^2(T) g^2(X_t) + D_{\kappa,C}^2(T)\right)\,dt
\end{align*}
so that
\begin{equation*}
  \pQV{M^T}_T \leq 8 \sigma^2 C_{\kappa}^2(T) \int_0^T g^2 (X_t) \, dt +
  8 \sigma^2 D_{C,\kappa}^2(T)\, T
\end{equation*}
with $C_{\kappa}(T), D_{C,\kappa}(T)$ as in the statement of the Proposition.
The result now follows from Corollary~\ref{corr:selfnorm}.
\end{proof}

\begin{remark}
  If $X$ is a diffusion in the sense that
  $\Gamma(\Phi(f_1), f_2) = \Phi'(f_1) \Gamma(f_1, f_2)$ for all
  continuously differentiable $\Phi$ and $f_1, f_2$ in the domain of
  $\Gamma$, then for all continuously differentiable functions $g$
  such that $2\Gamma(g) \leq 1$, we have
  $\Gamma(g^2) = (2g)^2 \Gamma(g) \leq 2 g^2$. In particular, if
  $\Gamma(g) = \frac12 \norm{\nabla g}^2$, then this holds when $g$ is
  differentiable and 1-Lipschitz.
\end{remark}

The preceding proposition applies in particular to the
Ornstein-Uhlenbeck process on $\R^d$ with $g(x)^2 =
\abs{x}^2$. Indeed, for $\kappa > 0$, consider the
$d$-dimensional Ornstein-Uhlenbeck process with generator
\begin{equation*}
  Lf(x) = -\kappa x \cdot \nabla f + \frac12 \Delta f.
\end{equation*}
By a direct calculation we have $Lg^2 = -2\kappa g^2 + d$. The
corresponding squared field operator is
$\Gamma(f) = \frac12 \abs{\nabla f}^2$ so that $\Gamma(g^2) = 2 g^2$.
We also have
\begin{equation*}
\Gamma(P_{s,t} f) \leq \left(P_{s,t}(\sqrt{\Gamma(f)})\right)^2 e^{-2\kappa (t-s)},
\end{equation*}
see for example~\autocite{bakry_analysis_2014,}. From the expression
for $Lg^2$ we get furthermore
$\partial_t P_{0,t} g^2 = P_{0,t} Lg^2 = -2\kappa P_{0,t} g^2 + d$ so
that for $X_0 = 0$ we have
$\E g^2(X_t) = P_{0,t} g^2(X_0) = d/(2\kappa)\, (1-e^{-2\kappa t})$
and $\E S_T = \int_0^T \E g^2(X_t)\,dt \leq d/(2\kappa) T$.  In the
notation of the preceding Proposition, we have
$C_\kappa \leq 1/(2\kappa)$, $D_{\kappa,C} \leq d/(\kappa^2)$ so that
\begin{equation*}
  \PP\left(\Abs{\frac1T \int_0^T \abs{X_t}^2\,dt - \E\left[\frac1T \int_0^T \abs{X_t}^2\,dt\right]} \geq R \right) \leq 2^{1/3}
  \exp\left(-\frac{\kappa^2 R^2 T}{6\left(R + d/\kappa + 8 d^2/(\kappa^2)\right)}\right).
\end{equation*}
This bound is looser than the one obtained in
Proposition~\ref{prop:OU2} above, which relies on the special property
of the Ornstein-Uhlenbeck process~\eqref{eq:OU2nabla} to obtain the
tighter bound.

\subsection{SDEs with degenerate diffusion matrix}\label{sec:examples:degenerate}
For $\alpha, \beta > 0$ let $(X^{x,y}, Y^{x,y})_{x,y\in\R}$ be the
family of solutions to
\begin{align*}
dX^{x,y}_t &= -\alpha X^{x,y}_t dt + dB_t, \quad X^{x,y}_0 = x, \\
dY^{x,y}_t &= -\beta Y^{x,y}_t dt + X^{x,y}_t dt, \quad Y^{x,y}_0 = y.
\end{align*}
In other words, $X^{x,y}$ is an
Ornstein-Uhlenbeck process and $Y^{x,y}_t$ can be
written $Y^{x,y}_t = y \, e^{-\beta t} + \int_0^t e^{-\beta (t-s)} \,X^x_s\,ds$.  The
associated semigroup is $P_t f(x, y) = \E f(X^{x,y}_t, Y^{x,y}_t)$ and
the squared field operator is
\begin{equation*}\Gamma(f)(x, y) =
  \abs{\partial_x f(x, y)}^2.
\end{equation*}
The process admits an invariant measure by classical results from~\autocite{meyn_stability_1993,}.
We note that no Poincaré inequality with respect to the invariant measure
and $\Gamma$ can hold for any class of functions containing functions $f(x, y)$ such that $\partial_x f = 0$ and $\partial_y f \ne 0$. 

In particular, since $Y$ has differentiable trajectories, this
example illustrates that our concentration inequalities do not rely on
``roughness'' of the trajectories.

\begin{proposition}
For all 1-Lipschitz functions $f$ and $R, T > 0$
\begin{multline*}
\PP\left(\frac1T \int_0^T f(Y^{x,y}_t) \,dt - \E\left[\frac{1}T
    \int_0^T f(Y^{x,y}_t)\,dt \right] \geq R\right)  \\
\leq \exp\left(\frac{-R^2 \, T \, (\alpha \wedge \beta)^2 \, \abs{\alpha - \beta}^2}{4 \left(1-e^{-(\alpha\wedge\beta)T}\right)^2 \left(1-e^{-\abs{\alpha-\beta}T}\right)^2}\right).
\end{multline*}
\end{proposition}

\begin{proof}
We have by Itô's formula that
\begin{align*}
e^{\beta t} Y^{x,y}_t
&= y + \int_0^t e^{(\beta-\alpha) s} e^{\alpha s} X^{x,y}_s ds \\
&= y + x \int_0^t e^{(\beta-\alpha) s} ds + \int_0^t e^{(\beta-\alpha) s} \left( \int_0^s e^{\alpha r} dB_r \right) ds
\end{align*}
so that for all $h, x, y \in \R$, by cancelling terms that don't
depend on $x$, integrating and rearranging,
\begin{equation*}
\frac{Y^{x+h,y} - Y^{x,y}}h 
= \frac{e^{-\alpha t} - e^{-\beta t}}{\beta - \alpha} 
= \frac{e^{-(\alpha \wedge \beta)t} \, (1-e^{-\abs{\alpha-\beta}t})}{\abs{\alpha - \beta}}.
\end{equation*}

In particular, for a 1-Lipschitz function $f(y)$, using the Lipschitz
property and the bound from above
\begin{align*}
\sqrt{\Gamma(P_t f)(x, y)} &= \abs{\partial_x P_t f(x, y)} 
= \lim_{h\to 0} \frac1{\abs{h}} \abs{\E f(Y^{x+h,y}_t) - f(Y^{x,y}_t)} \\
&\leq \lim_{h\to 0} \E \frac{\abs{Y^{x+h,y}_t - Y^{x,y}_t}}{\abs{h}} \\
&= \frac{e^{-(\alpha \wedge \beta)t} \, (1-e^{-\abs{\alpha-\beta}t})}{\abs{\alpha - \beta}}.
\end{align*}
By Proposition \ref{prop:Markov}, time-homogeneity of $P_t$ and the
Cauchy-Schwartz inequality for $\Gamma$
\begin{align*}
\pQV{M^T}_T &= \int_0^T \int_t^T \int_t^T \Gamma(P_{t,u}f, P_{t,v}
              f)(X^{x,y}_t, Y^{x,y}_t)\,du\,dv\,dt \\
& \leq 2 \int_0^T \left(\int_0^{T-t} \sqrt{\Gamma(P_u f)(X^{x,y}_t, Y^{x,y}_t)}\, du\right)^2 dt.
\end{align*}

We have from the bound on $\sqrt{\Gamma(P_t f)(x, y)}$ derived above
\begin{align*}
\int_0^{T-t} \sqrt{\Gamma(P_u f) (X^{x,y}_t, Y^{x,y}_t)} du
& \leq \int_0^T e^{-(\alpha \wedge \beta)u} du \, \frac{(1-e^{-\abs{\alpha-\beta}T})}{\abs{\alpha - \beta}} \\
& \leq  \frac{\left(1-e^{-(\alpha\wedge\beta)T}\right)}{(\alpha \wedge \beta)}  \frac{\left(1-e^{-\abs{\alpha-\beta}T}\right)}{\abs{\alpha - \beta}}
\end{align*}
so that finally
\begin{equation*}
\pQV{M^T}_T \leq 2 T \frac{\left(1-e^{-(\alpha\wedge\beta)T}\right)^2}{(\alpha \wedge \beta)^2}  \frac{\left(1-e^{-\abs{\alpha-\beta}T}\right)^2}{\abs{\alpha - \beta}^2}
\end{equation*}
and the result follows.
\end{proof}

\section*{Funding}
This work was supported by the National Research Fund, Luxembourg.
\printbibliography

\end{document}